\documentclass[11pt,a4paper]{article}

\usepackage{amssymb}
\usepackage{latexsym}
\usepackage{amsmath}
\usepackage{graphicx}
\usepackage{amsthm}
\usepackage{empheq}
\usepackage{bm}
\usepackage{booktabs}
\usepackage[dvipsnames]{xcolor}
\usepackage{pagecolor}
\usepackage{caption}
\usepackage{subcaption}
\usepackage{tikz,pgfplots}

\usepackage[margin=2.7cm]{geometry}

\numberwithin{equation}{section}

\theoremstyle{definition}
\newtheorem{theorem}{Theorem}[section]
\newtheorem{corollary}[theorem]{Corollary}
\newtheorem{proposition}[theorem]{Proposition}
\newtheorem{definition}[theorem]{Definition}
\newtheorem{example}[theorem]{Example}

\newtheorem{notation}[theorem]{Notation}
\newtheorem{remark}[theorem]{Remark}

\newtheorem{lemma}[theorem]{Lemma}

\newcommand{\numberset}{\mathbb}
\newcommand{\N}{\numberset{N}}
\newcommand{\Z}{\numberset{Z}}

\newcommand{\R}{\numberset{R}}

\newcommand{\F}{\numberset{F}}

\newcommand{\mC}{\mathcal{C}}
\newcommand{\mS}{\mathcal{S}}
\newcommand{\mG}{\mathcal{G}}

\newcommand{\mF}{\mathcal{F}}

\newcommand{\rk}{\textnormal{rk}}

\newcommand{\mat}{\operatorname{Mat}_q^{n \times m}}
\newcommand{\matalt}{\operatorname{Alt}_q^{n \times n}}
\newcommand{\matsym}{\operatorname{Sym}_q^{n \times n}}

\newcommand{\supp}{\textnormal{supp}}

\newcommand{\inv}{\textnormal{inv}}

\newcommand{\diag}{\Delta}
\newcommand{\NAR}{\textup{NAR}}
\newcommand{\sym}{\textup{sym}}
\newcommand{\alt}{\textup{alt}}

\newtheorem{claim}{\textit{Claim}}

\newcommand*{\myproofname}{Proof of the claim}
\newenvironment{clproof}[1][\myproofname]{\begin{proof}[#1]}{\end{proof}}

\usetikzlibrary{positioning,chains,fit,shapes,calc}

\def\atk{\textup{ATK}}
\def\lex{\textup{lex}}

\usepackage{chessfss}
\usepackage{adjustbox}

\usepackage{hyperref}

\title{\textbf{The Diagonals of a Ferrers Diagram}}

\usepackage{authblk}
\usepackage{comment}

\author[1]{Giuseppe Cotardo\thanks{G. C. is partially supported by the NSF grants DMS-2037833 and DMS-2201075, and by the Commonwealth Cyber Initiative.}}
\affil[1]{Virginia Tech, U.S.A.}

\author[2]{Anina Gruica\thanks{A. G. is supported by the Dutch Research Council through grant OCENW.KLEIN.539.}}
\affil[2]{Eindhoven University of Technology, the Netherlands}

\author[2]{Alberto Ravagnani\thanks{A. R. is supported by the Dutch Research Council through grants VI.Vidi.203.045, 
OCENW.KLEIN.539, 
and by the Royal Academy of Arts and Sciences of the Netherlands.}}

\date{}

\usepackage{setspace}
\allowdisplaybreaks
\begin{document}

\maketitle
\thispagestyle{empty}
	
\begin{abstract}
We propose and develop a theory of Ferrers diagrams and their $q$-rook polynomials solely based on their diagonals. We show that the cardinalities of the diagonals of a Ferrers diagram are equivalent information 
to their rook numbers, $q$-rook polynomials, and the rank distribution of matrices supported on the diagram.
Our approach is based on the concept of \textit{canonical form} of a Ferrers diagrams, and  on two simple diagram operations as the main proof tools.
In the second part of the paper we develop the same theory for symmetric Ferrers diagrams, considering symmetric and alternating matrices supported on them. As an application of our results, we establish some combinatorial identities linking symmetric and alternating matrices, which do not appear to have an obvious bijective proof, and which generalize some curious results in enumerative combinatorics. 
\end{abstract}

\medskip

\section*{Introduction}

Rook theory studies 
the ways in which a certain number of non-attacking rooks can be placed on a given subset of the chess board; see e.g.~\cite{riordan1958introduction} for a general reference. The number of such placements depends on the number of rooks one wishes to place and on the shape of the board's subset. Ferrers diagrams are special subsets of the board
that have a ``staircase'' shape. Their rook theory has been extensively studied, also in connection with the number of matrices over a finite field~$\F_q$ whose support is contained in the diagram; see~\cite{haglund2001rook, lewis2010matrices, klein2014counting,gruica2022rook,gluesing2020partitions}. The study of this connection is often called
``$q$-rook theory'' and it is the combinatorial $q$-analogue of (classical) rook theory. Ferrers diagrams appear to behave particularly well from a $q$-rook theory perspective. For instance, the number of matrices over $\F_q$ of any given rank supported on a Ferrers diagram is a polynomial in $q$, a property that is false for general diagrams; see for instance~\cite[Section 8]{stembridge1998counting}. Skew Ferrers diagrams behave somewhat similarly to Ferrers diagrams, although their theory is substantially more complex~\cite{lewis2020rook}.

This paper shows that the $q$-rook theory of Ferrers diagrams
is linked to the combinatorics of their diagonals in a very natural way. 
In fact, it follows from our results that the $q$-rook polynomials of a Ferrers diagrams are equivalent information to the cardinalities of their diagonals. As a corollary, we obtain that to check whether or not two Ferrers diagrams have the same ($q$-)rook numbers it suffices to check if their diagonals have the same cardinalities. 
Clearly, the cardinalities of the diagonals of a Ferrers diagram do not fully specify the diagram itself: our paper shows that they 
nonetheless retain the relevant combinatorial information.

Deriving the $q$-rook theory of Ferrers diagrams via their diagonals comes with some other interesting advantages, which we briefly survey in the remainder of this introduction, offering an overview of the contributions made by this paper.

Focusing on the diagonals of Ferrers diagrams allows us to 
introduce a notion of equivalence between diagrams. We call Ferrers diagrams \textit{diagonally equivalent} if their diagonals have the same cardinalities.
We then identify two operations
that take a diagram as input and return a smaller diagram. We call these operations \textit{puncturing} and \textit{reduction}, and study them in Section~\ref{sec:2}.
Puncturing and reduction preserve diagonal equivalence of Ferrers diagrams. In particular, they can be efficiently used in recursive arguments that establish the properties of equivalent diagrams. That's precisely the approach we take 
in Section~\ref{sec:3}. Our main result in this context is Theorem~\ref{thm:bijection}
 and its consequences, stated in Corollary~\ref{cor:surprise}.

 In the second part of the paper we study the theory of Ferrers diagrams and their diagonals in connection with alternating and symmetric matrices supported on them. This naturally leads to the notion of \textit{symmetric} Ferrers diagram; see Section~\ref{sec:4}.
 We introduce notions of $q$-rook polynomials for symmetric Ferrers diagrams that capture the combinatorics of alternating and symmetric matrices over $\F_q$. This revises and partially extends the approach of~\cite{lewis2010matrices}.

 In the last section of the paper, we apply the theory of $q$-rook polynomials to derive curious identities involving symmetric and alternating matrices at the same time. It is well known that these two classes of matrices exhibit similarities that are hard to explain via bijective proofs; see e.g.~\cite[Section 3]{lewis2020rook}. $q$-Rook polynomials turn out to be powerful tools to establish these identities. Our formulas generalize a result of~\cite{lewis2010matrices} from unrestricted matrices to matrices supported on a Ferrers diagram.

\section{Ferrers Diagrams}
We start by recalling the main definitions and classical results that will be needed throughout this paper. In the sequel, $q$ denotes a prime power and $\F_q$ is the finite field with $q$ elements. Depending on the context, $q$ will sometimes also denote the indeterminate of a univariate polynomial.

We let $m$ and $n$ be positive integers and denote by $\mat$ the $\F_q$-linear space of $n \times m$ matrices over $\F_q$. We denote by $\N$ the set of positive integers and by $\N_0$ the set of non-negative integers.
We call $\N \times \N$
the \textbf{infinite board}, which we depict as follows:
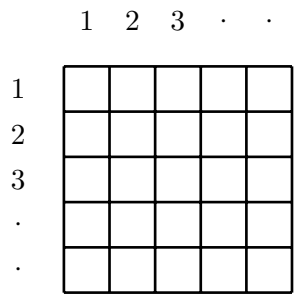
\begin{figure}[h]
    \centering
    \begin{tikzpicture}[scale=0.6, line width=1pt]
    \draw (0,0) grid (5,5);
    \node at (-1,4.5) {1};
    \node at (-1,3.5) {2};
    \node at (-1,2.5) {3};
    \node at (-1,1.5) {$\cdot$};
    \node at (-1,0.5) {$\cdot$};
    \node at (4.5,6) {$\cdot$};
    \node at (3.5,6) {$\cdot$};
    \node at (2.5,6) {3};
    \node at (1.5,6) {2};
    \node at (0.5,6) {1};
    \end{tikzpicture}
\end{figure}

This paper is primarily concerned with \textit{Ferrers diagrams}, defined in the following way.

\begin{definition}
A \textbf{Ferrers diagram} is a finite subset $\mF \subseteq \N \times \N$ with the property that if $(i,j)\in\mF$, then $(s,t)\in\mF$ for all $s\in \{1,\ldots,i\}$ and $t\in \{1,\ldots,j\}$. We identify a Ferrers diagram with its \textbf{column lengths} $[c_1, c_2 \dots,c_m]$, where $m=\max\{j : i \in \N, (i,j) \in \mF\}$ and $c_j=|\{(i,j) : i\in\N, \, (i,j) \in \mF \}|$
for all $j \in \{1, \dots, m\}$.
\end{definition}

One can visually represent a Ferrers diagram as in Figure~\ref{fig:F5332} with the typical staircase shape.
Note that our diagrams are top-left aligned.

\begin{figure}[h]
\centering
  \begin{tikzpicture}[scale=0.6, line width=1pt]
        \draw (0,0) grid (5,-5);
         \foreach \n in {0,...,3} {\draw[draw=black, fill =Goldenrod!55] (0,-\n) rectangle (1,-1-\n);}
         \foreach \n in {0,...,2} {\draw[draw=black, fill =Goldenrod!55] (1,-\n) rectangle (2,-1-\n);}
         \foreach \n in {0,...,2} {\draw[draw=black, fill =Goldenrod!55] (2,-\n) rectangle (3,-1-\n);}
         \foreach \n in {0,...,1} {\draw[draw=black, fill =Goldenrod!55] (3,-\n) rectangle (4,-1-\n);}
          \foreach \n in {0,...,0} {\draw[draw=black, fill =Goldenrod!55] (4,-\n) rectangle (5,-1-\n);}
    \end{tikzpicture}
    \caption{The Ferrers diagram $\mF=[4,3,3,2,1]$. \label{fig:F5332}} 
\end{figure}
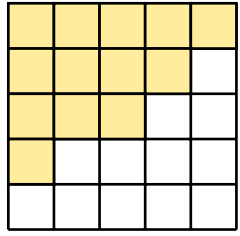

Following Haglund's work~\cite{haglund1998q}, in this paper we study Ferrers diagrams 
 in connection with the matrices over $\F_q$ that are \textit{supported} on it.

\begin{definition}
The \textbf{support} of  $M \in \mat$ is $\operatorname{supp}(M) = \{(i,j) : M_{ij} \ne 0\} \subseteq \{1,\ldots,n\} \times \{1,\ldots,m\}$. For a Ferrers diagram $\mF \subseteq \{1,\ldots,n\} \times \{1,\ldots,m\}$, we denote by $\mat[\mF]$ the space of matrices $M \in \mat$ 
with $\supp(M) \subseteq \mF$.
\end{definition}

Clearly, $\mat[\mF]$ has dimension $|\mF|$ over $\F_q$. We also recall the following definitions from rook theory.

\begin{definition}
A \textbf{placement} of \textbf{non-attacking rooks} is a subset $P \subseteq \N \times \N$ with the property that no two elements of $P$ have the same row or column index. The elements of $P$ are called \textbf{rooks}. For a Ferrers diagram $\mF \subseteq\N \times \N$ and for $r \ge 0$, we denote by $\textup{NAR}(\mF,r)$ the set of non-attacking rook placements $P$ with $P \subseteq \mF$ and $|P| = r$.
\end{definition}

\begin{definition}
The \textbf{degree} of a Ferrers diagram $\mF$ 
is $\partial(\mF)=\max\{r : \NAR(\mF,r) \neq \emptyset\}$.
\end{definition}

Note that $\partial(\mF)$ is the degree of the classical rook polynomial associated with $\mF$, which explains our choice for the name.
It turns out that $\partial(\mF)$ coincides with the maximum rank a matrix supported on $\mF$ can have, independently on the field size $q$. This fact is well known, but we include a short proof for completeness.

\begin{proposition}\label{prop:partialFmaxrk}
     We have $\partial(\mF)=\max\{\rk(M) : M \in \mat[\mF]\}$ for any Ferrers diagram $\mF \subseteq \{1, \ldots, n\} \times \{1, \ldots, m\}$.
\end{proposition}
\begin{proof}
Suppose that $P \in \NAR(\mF,r)$. Construct a matrix $M$ by setting
$M_{ij}=1$ if $(i,j) \in P$ and $M_{ij}=0$ otherwise. Then $M$ is supported on $\mF$ and has rank $r$. Vice versa,
by the K\"onig-Egev\'ary Theorem (see~\cite[Chapter~7]{riordan1958introduction}) $\partial(\mF)$ is the minimum number of lines (i.e., rows or columns) that are needed to cover all the elements of $\mF$. Therefore the rank of any matrix~$M$ supported on $\mF$ cannot exceed $\partial(\mF)$.
\end{proof}

In~\cite{haglund1998q}, Haglund establishes a connection between the number of matrices supported on a Ferrers diagram~$\mF$ and the \textit{$q$-rook polynomials} of $\mF$.
These have been defined by Garsia and Remmel
in~\cite{GaRe86} and are based on the following 
numerical invariants.

\begin{notation}
\label{not:inv}
Let $\mF\subseteq \N \times \N$ be a Ferrers diagram
and let $P \subseteq \mF$ be a placement of non-attacking rooks. We denote by $\inv(\mF,P)$ the number computed as follows: cross out all the cells from~$\mF$ that either correspond to a rook of $P$, or are above or to the left of any rook of $P$. Then $\inv(\mF,P)$ is the number of cells
of $\mF$ not crossed out. In the rook theory context,
$\inv(\mF,P)$ is called a \textit{statistics}.
\end{notation}

For example,
consider the Ferrers diagram $\mF=[4,3,3,2,1]$ and the placement of non-attacking rooks $P=\{(2,4), (3,2),(4,1)\}$. 
These are depicted in Figure~\ref{fig:F5332NAR},
where the rooks are marked with ``$\rook$'' and the cancellations that are performed when computing $\inv(\mF,P)$, excluding the rooks, are marked with ``$\color{Red}\times$''. We have $\inv(\mF,P)=3$.

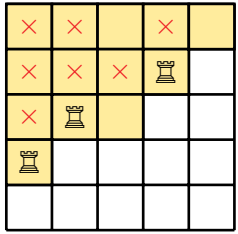
\begin{figure}[h]
\centering
    \begin{tikzpicture}[scale=0.6, line width=1pt]
        \draw (0,0) grid (5,-5);
         \foreach \n in {0,...,3} {\draw[draw=black, fill =Goldenrod!55] (0,-\n) rectangle (1,-1-\n);}
         \foreach \n in {0,...,2} {\draw[draw=black, fill =Goldenrod!55] (1,-\n) rectangle (2,-1-\n);}
         \foreach \n in {0,...,2} {\draw[draw=black, fill =Goldenrod!55] (2,-\n) rectangle (3,-1-\n);}
         \foreach \n in {0,...,1} {\draw[draw=black, fill =Goldenrod!55] (3,-\n) rectangle (4,-1-\n);}
          \foreach \n in {0,...,0} {\draw[draw=black, fill =Goldenrod!55] (4,-\n) rectangle (5,-1-\n);}
         \node at (3.5,-1.5) {\color{black}\rook};
         \node at (3.5,-0.5) {$\color{Red}\times$};
         \node at (2.5,-1.5) {$\color{Red}\times$};
         \node at (1.5,-2.5) {\color{black}\rook};
         \node at (1.5,-1.5) {$\color{Red}\times$};
         \node at (1.5,-0.5) {$\color{Red}\times$};
         \node at (0.5,-3.5) {\rook};
         \node at (0.5,-2.5) {$\color{Red}\times$};
         \node at (0.5,-1.5) {$\color{Red}\times$};
         \node at (0.5,-0.5) {$\color{Red}\times$};
    \end{tikzpicture}
    \caption{A placement $P$ of $3$ non-attacking rooks on the Ferrers diagram~$\mF=[4,3,3,2,1]$ and the cancellations performed to compute $\inv(\mF,P)$.\label{fig:F5332NAR}}
    \end{figure}

The $q$-rook polynomials of a Ferrers diagram are defined as follows.

\begin{definition} \label{def:qrook}
Let $\mF$ be a Ferrers diagram and let $r\in\Z$. The $r$\textbf{-th} $q$\textbf{-rook polynomial} of $\mF$ is
\begin{equation*}
    R_r(\mF;q)=\sum_{P\in\textup{NAR}(\mF,r)}q^{\inv(\mF,P)}\in\Z[q].
\end{equation*}
We set $R_r(\mF;q)=0$ if $r<0$ or $\textup{NAR}(\mF,r)=\emptyset$, and $R_0(\mF;q)=q^{|\mF|}$.
\end{definition}

The polynomials of Definition~\ref{def:qrook} 
are the $q$-analogues of the classical rook polynomials; see e.g.~\cite{riordan1958introduction,kaplansky1946problem}. A remarkable fact (see~\cite[page~257]{GaRe86}) is that Ferrers diagrams have the same classical rook polynomial if and only if they have the same $q$-rook polynomials.

In this paper, we will often need to count the number of matrices of given rank in an $\F_q$-linear space. We introduce the following notation to make the statements more compact.

\begin{notation}
For an $\F_q$-linear space of matrices $\mC \le \mat$ and an integer $r \in \N_0$, we let
$W_r(\mC)=|\{M \in \mC : \rk(M)=r\}|$.
\end{notation}

We can finally state Haglund's theorem connecting rook theory and the distribution of matrices supported on a Ferrers diagram.

\begin{theorem}[{\cite[Theorem~1]{haglund2001rook}}] \label{thm:hag}
Let $\mF$ be a Ferrers diagram and let $r\in\N_0$. We have
\begin{align*}
    W_r(\mat[\mF]) = (q-1)^r q^{|\mF|-r}R_r(\mF;q^{-1}).
\end{align*}
\end{theorem}

\section{
Operations on Ferrers Diagrams} \label{sec:diag}\label{sec:2}

The classical theory of Ferrers diagrams is mostly centered around the properties of their rows and columns. In this paper, we 
show that the theory of Ferrers diagrams
can be reconstructed entirely 
by looking only at their \textit{diagonals}.
This gives rise to an interesting notion of equivalence of Ferrers diagrams:
we declare two diagrams equivalent if they have the same diagonal cardinalities. We then show that equivalent diagrams have the same $q$-rook polynomials, which has surprising consequences for the behaviour of matrices over~$\F_q$ supported on a Ferrers diagram.

The main results are stated and proved in Section~\ref{sec:equiv}. This section contains a series of technical results that we will need in the proofs. Our approach is based on two operations on diagrams (\textit{puncturing} and \textit{reduction}), which 
preserve equivalence.

\begin{definition}
Let $i\in\N$. Then the $i$\textbf{-th diagonal} of the board $\N\times \N$ is the set $\diag_i=\{(r,i-r+1):r\in\N\}$. We say that Ferrers diagrams $\mF,\mF' \subseteq \N\times \N$ 
are \textbf{diagonally equivalent} if $|\diag_i\cap\mF|=|\diag_i\cap\mF'|$ for all $i\in\N$, i.e., if the cardinalities of their diagonals are the same.
\end{definition}

The following is a visual representation of
the five diagonals of the board $\{1,\ldots,5\}\times\{1,\ldots,5\}$, each of which has a different color: 

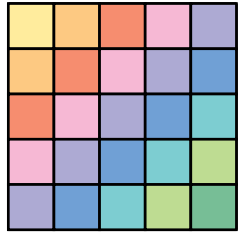
\begin{center}
    \begin{tikzpicture}[scale=0.6, line width=1pt]
    \draw (0,1) grid (5,-4);
    \draw[draw=black, fill =Goldenrod!55] (0,1) rectangle (1,0);
    \foreach \n in {0,...,1} {\draw[draw=black, fill =YellowOrange!55]  (2-\n,1-\n) rectangle (1-\n,0-\n);}
    \foreach \n in {0,...,2} {\draw[draw=black, fill =Red!55]  (3-\n,1-\n) rectangle (2-\n,0-\n);}
    \foreach \n in {0,...,3} {\draw[draw=black, fill =CarnationPink!55]  (4-\n,1-\n) rectangle (3-\n,0-\n);}
     \foreach \n in {0,...,4} {\draw[draw=black, fill =Periwinkle!55]  (5-\n,1-\n) rectangle (4-\n,0-\n);}
     \foreach \n in {1,...,4} {\draw[draw=black, fill =RoyalBlue!55]  (6-\n,1-\n) rectangle (5-\n,0-\n);}
     \foreach \n in {1,...,3} {\draw[draw=black, fill =TealBlue!55]  (6-\n,0-\n) rectangle (5-\n,-1-\n);}
      \foreach \n in {1,...,2} {\draw[draw=black, fill =LimeGreen!55]  (6-\n,-1-\n) rectangle (5-\n,-2-\n);}
     \foreach \n in {1,...,1} {\draw[draw=black, fill =ForestGreen!55]  (6-\n,-2-\n) rectangle (5-\n,-3-\n);}
    \end{tikzpicture}
\end{center}

We will see throughout the paper that the cardinalities of the intersections between a Ferrers diagram and the board's diagonals retain a wealth of information about
the diagram. For instance, its degree can be easily computed as follows.

\begin{proposition}\label{prop:dFdiag}
We have  $\partial(\mF)=\max\{|\diag_i\cap\mF| : i\in\N\}$ for any Ferrers diagram $\mF \subseteq \{1, \ldots, n\} \times \{1, \ldots, m\}$.
\end{proposition}
\begin{proof}
Proposition~\ref{prop:partialFmaxrk} and~\cite[Claim~A]{gruica2022rook}
    give $\partial(\mF)\in\{|\diag_i\cap\mF| : i\in\N\}$, which implies $\partial(\mF)\leq\max\{|\diag_i\cap\mF| : i\in\N\}$. On the other hand, let $r=\max\{|\diag_i\cap\mF| : i\in\N\}$ and define the matrix $M\in\mat$ by $M_{ij}=1$ if $(i,j)\in\diag_r\cap\mF$ and $M_{ij}=0$ otherwise. Then~$M$ is 
    supported on $\mF$ and $\rk(M)=r$. It follows that $\partial(\mF)\geq r$ by Proposition~\ref{prop:partialFmaxrk}.
\end{proof}

We start by labelling the elements of the board $\N\times\N$ based on the diagonal to which they belong.

\begin{notation}
\label{not:lex}
Let $\mF$ be a Ferrers diagram, $i\in\N$, and $j\in \{1,\ldots,i\}$. We denote by 
$(\Delta_i \cap \mF)_j$ the $j$-th element of
$\Delta_i \cap \mF$
 starting from the top-most row. We then sort the elements of the set
\begin{equation*}
    \{(\Delta_i \cap \mF)_j \, : \, i\in\N, \, j\in\{1,\ldots,|\diag_i \cap \mF|\}\}
\end{equation*}
according to the lexicographic order $<_\lex$ with respect to the pair~$(i,j)$. We denote the order on $\N \times \N$ simply by $\le$.
\end{notation}

Figure~\ref{fig:diag4} shows the fourth diagonal of $\mF=\{1,\ldots,5\}\times\{1,\ldots,5\}$ enumerated according to Notation~\ref{not:lex}. In Figure~\ref{fig:points} we display 
the elements $A=(\diag_3\cap\mF)_2=(2,2)$ and $B=(\diag_5\cap\mF)_5=(5,1)$. Note that $A<B$, since $(3,2)<_\lex(5,5)$.

\begin{figure}[h]
    \begin{minipage}{0.49\textwidth}
    \centering
    \begin{tikzpicture}[scale=0.6, line width=1pt]
    \draw (0,1) grid (5,-4);
    \foreach \n in {0,...,3} {\draw[draw=black, fill =CarnationPink!55]  (4-\n,1-\n) rectangle (3-\n,0-\n);}
    \node at (3.5,0.5) {$1$};
    \node at (2.5,-0.5) {$2$};
    \node at (1.5,-1.5) {$3$};
    \node at (0.5,-2.5) {$4$};
    \end{tikzpicture}
    \captionof{figure}{Sorted elements in $\diag_4\cap\mF$.}
    \label{fig:diag4}
    \end{minipage}\hfill
    \begin{minipage}{0.49\textwidth}
    \centering
    \begin{tikzpicture}[scale=0.6, line width=1pt]
    \draw (0,1) grid (5,-4);
    \draw[draw=black, fill =Red!55]  (2,-1) rectangle (1,0);
    \draw[draw=black, fill =Periwinkle!55]  (1,-4) rectangle (0,-3);
    \node at (1.5,-.5) {$A$};
    \node at (.5,-3.5) {$B$};
    \end{tikzpicture}
    \captionof{figure}{Representation of $A$ and $B$.}
    \label{fig:points}
    \end{minipage}
\end{figure}

\begin{remark} \label{rem:power}
The diagonals of a Ferrers diagram partition its entries $(a,b)$ according to the value of $\inv(a,b)$; recall Notation~\ref{not:inv}.
Therefore, and more formally,
\begin{align*}
    \diag_i\cap\mF = \{(a,b) \in \mF : \inv(\mF,(a,b)) = |\mF|-i\}.
\end{align*}
This means that the diagonal a cell of $\mF$ 
belongs to determines how many cells are crossed out when computing $\inv$ as in Notation~\ref{not:inv}. 
\end{remark}

\begin{notation}\label{not:atk}
For any subset $P \subseteq \N\times\N$, we denote by $\atk(P)$ the set of cells deleted by the elements of $P$ according to the statistic inv, namely
\begin{equation*}
\atk(P)=\bigcup_{(i,j)\in P}\{(1,j),\ldots,(i,j),(i,j-1),\ldots,(i,1)\}.
\end{equation*}
Moreover, we say that the rooks in $P$ \textbf{attack} the cells in $\atk(P)$.
\end{notation}

Note that for a Ferrers diagram $\mF$, $r \in \N$, and $P \in \NAR(\mF,r)$, we have $|\atk(P)|=|\mF|-\inv(\mF,P)$.

\begin{example}
    Let $P$ be the placement of non-attacking rooks on the Ferrers diagram $\mF$ as in 
    Figure~\ref{fig:F5332NAR}. One can check that
    \begin{equation*}
        \atk(P)=\{(1,1),(1,2),(1,4),(2,1),(2,2),(2,3),(2,4),(3,1),(3,2),(4,1)\}.
    \end{equation*}
    Moreover, $|\mF|=13$, $\inv(\mF;P)=3$ and $|\atk(P)|=10=|\mF|-\inv(\mF,P)$.
\end{example}

We continue by introducing
some operations that allow us to obtain Ferrers sub-diagrams of a Ferrers diagram.

\begin{definition} \label{def:reduction}
Let $\mF$ be a Ferrers diagram. The \textbf{puncturing} $\widehat\mF \subseteq \mF$ is the Ferrers sub-diagram of $\mF$ obtained by removing $\diag_r\cap\mF$, where $r:=\max\{s\in\N:\diag_s\cap\mF\neq\emptyset\}$. For $i\in\N$ and $j\in \{1,\ldots,i\}$, the \textbf{reduction} of $\mF$ with respect to $(\Delta_i \cap F)_j$ is the Ferrers sub-diagram of $\mF$ obtained by removing the set
\begin{equation*}
     \atk(\{(\diag_i\cap\mF)_j\})\cup\{(\diag_t\cap\mF)_s:(i,j)<_\lex (s,t)\}
\end{equation*}
from $\mF$ and aligning the remaining cells to the top and then to the left. We denote this new Ferrers diagram by $\Pi(\mF,i,j)$. 
\end{definition}

\begin{example}
    Let $\mF$ be the Ferrers diagram of Figure~\ref{fig:F5332NAR}. The cells marked with ``\textcolor{red}{$\times$}'' in Figure~\ref{fig:punctcells} are the ones deleted by the operation of puncturing on $\mF$. The diagram $\widehat\mF$ resulting from this operation is shown in Figure~\ref{fig:punct}.
    \begin{figure}[h]
    \begin{minipage}{0.49\textwidth}
    \centering
    \begin{tikzpicture}[scale=0.6, line width=1pt]
     \draw (0,0) grid (5,-5);
    \foreach \n in {0,...,3} {\draw[draw=black, fill =Goldenrod!55] (0,-\n) rectangle (1,-1-\n);}
    \foreach \n in {0,...,2} {\draw[draw=black, fill =Goldenrod!55] (1,-\n) rectangle (2,-1-\n);}
    \foreach \n in {0,...,2} {\draw[draw=black, fill =Goldenrod!55] (2,-\n) rectangle (3,-1-\n);}
    \foreach \n in {0,...,1} {\draw[draw=black, fill =Goldenrod!55] (3,-\n) rectangle (4,-1-\n);}
    \foreach \n in {0,...,0} {\draw[draw=black, fill =Goldenrod!55] (4,-\n) rectangle (5,-1-\n);}
     \node at (4.5,-.5) {$\color{Red}\times$};
     \node at (3.5,-1.5) {$\color{Red}\times$};
     \node at (2.5,-2.5) {$\color{Red}\times$};
    \end{tikzpicture}
    \captionof{figure}{Cells of $\mF$ deleted by puncturing.}
    \label{fig:punctcells}
    \end{minipage}
    \begin{minipage}{0.49\textwidth}
    \centering
    \begin{tikzpicture}[scale=0.6, line width=1pt]
    \draw (0,1) grid (5,-4);
    \draw[draw=black, fill =Goldenrod!55] (0,1) rectangle (1,0);
    \foreach \n in {0,...,1} {\draw[draw=black, fill =Goldenrod!55]  (2-\n,1-\n) rectangle (1-\n,0-\n);}
    \foreach \n in {0,...,2} {\draw[draw=black, fill =Goldenrod!55]  (3-\n,1-\n) rectangle (2-\n,0-\n);}
    \foreach \n in {0,...,3} {\draw[draw=black, fill =Goldenrod!55]  (4-\n,1-\n) rectangle (3-\n,0-\n);}
    \end{tikzpicture}
    \captionof{figure}{The Ferrers diagram $\widehat\mF$.}
    \label{fig:punct}
    \end{minipage}
    \end{figure}
    Figure~\ref{fig:reduct} shows the reduction $\Pi(\mF,4,2)$ of $\mF$ with respect to $(\Delta_4 \cap \mF)_2$. In Figure~\ref{fig:reductcells}, we marked with ``\textcolor{red}{$\times$}'' the cells of $\mF$ deleted by this operation. 
    \begin{figure}[h]
    \begin{minipage}{0.49\textwidth}
    \centering
    \begin{tikzpicture}[scale=0.6, line width=1pt]
    \draw (0,1) grid (5,-4);
    \draw[draw=black, fill =Goldenrod!55] (0,1) rectangle (1,0);
    \foreach \n in {0,...,1} {\draw[draw=black, fill =Goldenrod!55]  (2-\n,1-\n) rectangle (1-\n,0-\n);}
    \foreach \n in {0,...,2} {\draw[draw=black, fill =Goldenrod!55]  (3-\n,1-\n) rectangle (2-\n,0-\n);}
    \foreach \n in {0,...,3} {\draw[draw=black, fill =Goldenrod!55]  (4-\n,1-\n) rectangle (3-\n,0-\n);}
     \foreach \n in {1,2,0} {\draw[draw=black, fill =Goldenrod!55]  (5-\n,1-\n) rectangle (4-\n,0-\n);}
     \node at (0.5,-0.5) {$\color{Red}\times$};
     \node at (1.5,-.5) {$\color{Red}\times$};
     \node at (2.5,.5) {$\color{Red}\times$};
     \node at (2.5,-0.5) {\rook};
     \node at (4.5,.5) {$\color{Red}\times$};
     \node at (1.5,-1.5) {$\color{Red}\times$};
     \node at (2.5,-1.5) {$\color{Red}\times$};
     \node at (0.5,-2.5) {$\color{Red}\times$};
      \node at (3.5,-0.5) {$\color{Red}\times$};
    \end{tikzpicture}
    \caption{Cells of $\mF$ deleted by reducing.}
    \label{fig:reductcells}
    \end{minipage}
    \begin{minipage}{0.49\textwidth}
    \centering
    \begin{tikzpicture}[scale=0.6, line width=1pt]
    \draw (0,1) grid (5,-4);
    \draw[draw=black, fill =Goldenrod!55] (0,1) rectangle (1,0);
    \foreach \n in {0,...,1} {\draw[draw=black, fill =Goldenrod!55]  (2-\n,1-\n) rectangle (1-\n,0-\n);}
    \foreach \n in {0,...,0} {\draw[draw=black, fill =Goldenrod!55]  (3-\n,1-\n) rectangle (2-\n,0-\n);}
    \end{tikzpicture}
    \caption{The Ferrers diagram $\Pi(\mF,4,2)$.}
    \label{fig:reduct}
    \end{minipage}
    \end{figure}
\end{example}

We continue with some technical results.

\begin{lemma} 
\label{lem:interreduc}
Let $\mF$ be a Ferrers diagram, 
$r \in \N$, and let 
\begin{equation*}
   P=\{(\diag_{i_1}\cap\mF)_{j_1},\ldots,(\diag_{i_r}\cap\mF)_{j_r}\}\in\NAR(\mF,r), 
\end{equation*}
be sorted according to the lexicographic order from smallest to largest. Then
\begin{equation*}
   P\setminus\{(\diag_{i_r}\cap\mF)_{j_r}\}\in\NAR(\Pi(\mF,i_r,j_r),r-1), 
\end{equation*}
up to realigning the cells 
as in Definition~\ref{def:reduction}.
\end{lemma}
\begin{proof}
The result is trivial for $r=1$, so suppose $r \ge 2$. Note that $(\diag_{i_r}\cap\mF)_{j_r}$ is the maximum in $P$ with respect to the lexicographic order. It follows that the cells in $P\setminus\{(\diag_{i_r}\cap\mF)_{j_r}\}$ cannot be in 
\begin{equation*}
    \atk((\diag_{i_r}\cap\mF)_{j_r})\cup\{(\diag_{s}\cap\mF)_{t}:s\in\N, t\in\{1,\ldots,|\diag_{s}\cap\mF|\},(i_r,j_r)<_\lex (s,t)\}.
\end{equation*}
This implies that $P\setminus\{(\diag_{i_r}\cap\mF)_{j_r}\}$ must be in $\Pi(\mF,i_r,j_r)$, up to cell realignment. The fact that $P\setminus\{(\diag_{i_r}\cap\mF)_{j_r}\}\in\NAR(\Pi(\mF,i_r,j_r),r-1)$, up to cell realignment, follows from the definition of reduction and the fact that $\atk(P\setminus\{(\diag_{i_r}\cap\mF)_{j_r}\}) \subseteq \mF$ contains $\atk(P\setminus\{(\diag_{i_r}\cap\mF)_{j_r}\})$ in $\Pi(\mF,i_r,j_r)$ as a subset, again up to cell realignment. 
\end{proof}

By Lemma~\ref{lem:interreduc}, we can extend the definition of reduction of a Ferrers diagram 
from one cell to multiple cells 
forming a placement of non-attacking rooks. 

\begin{definition}
Let $\mF$ be a Ferrers diagram and let $r \in\N$. Let $$P=\{(\diag_{i_1}\cap\mF)_{j_1},\ldots,(\diag_{i_r}\cap\mF)_{j_r}\} \in \NAR(\mF,r)$$ be ordered lexicographically from smallest to largest. We define $\Pi(\mF,P)$ iteratively as follows:
\begin{align*}
    \Pi(\mF,P) =  \begin{cases}
    \Pi(\mF,i_1,j_1) \quad &\textnormal{if $r=1$}, \\
    \Pi(\Pi(\mF,i_r,j_r),P \setminus \{(\diag_{i_r}\cap\mF)_{j_r}\}) \quad &\textnormal{otherwise.}
    \end{cases}
\end{align*}
\end{definition}

\begin{example}
Let $\mF$ be the Ferrers diagram of Figure~\ref{fig:F5332NAR}. Figure~\ref{fig:reductext} shows the reduction of~$\mF$ with respect to the placement of non-attacking rooks $P:=\{(\diag_4\cap\mF)_4,(\diag_4\cap\mF)_2\}$. The cells marked with ``\textcolor{red}{$\times$}'' in Figures~\ref{fig:reductext1} and~\ref{fig:reductext2} are the ones deleted by this operation.

	\begin{figure}[h]
    \begin{minipage}{0.33\textwidth}
    \centering
    \begin{tikzpicture}[scale=0.6, line width=1pt]
        \draw (0,0) grid (5,-5);
         \foreach \n in {0,...,3} {\draw[draw=black, fill =Goldenrod!55] (0,-\n) rectangle (1,-1-\n);}
         \foreach \n in {0,...,2} {\draw[draw=black, fill =Goldenrod!55] (1,-\n) rectangle (2,-1-\n);}
         \foreach \n in {0,...,2} {\draw[draw=black, fill =Goldenrod!55] (2,-\n) rectangle (3,-1-\n);}
         \foreach \n in {0,...,1} {\draw[draw=black, fill =Goldenrod!55] (3,-\n) rectangle (4,-1-\n);}
          \foreach \n in {0,...,0} {\draw[draw=black, fill =Goldenrod!55] (4,-\n) rectangle (5,-1-\n);}
         \node at (2.5,-1.5) {\rook};
        \node at (0.5,-3.5) {\rook};
        \node at (4.5,-.5) {$\color{Red}\times$};
        \node at (3.5,-1.5) {$\color{Red}\times$};
        \node at (2.5,-2.5) {$\color{Red}\times$};
        \node at (0.5,-0.5) {$\color{Red}\times$};
        \node at (0.5,-1.5) {$\color{Red}\times$};
        \node at (0.5,-2.5) {$\color{Red}\times$};
    \end{tikzpicture}
    \caption{$\mF$.}
    \label{fig:reductext1}
    \end{minipage}
    \begin{minipage}{0.33\textwidth}
    \centering
    \begin{tikzpicture}[scale=0.6, line width=1pt]
        \draw (1,0) grid (6,-5);
         \foreach \n in {0,...,2} {\draw[draw=black, fill =Goldenrod!55] (1,-\n) rectangle (2,-1-\n);}
         \foreach \n in {0,...,1} {\draw[draw=black, fill =Goldenrod!55] (2,-\n) rectangle (3,-1-\n);}
         \foreach \n in {0,...,0} {\draw[draw=black, fill =Goldenrod!55] (3,-\n) rectangle (4,-1-\n);}
         \node at (2.5,-1.5) {\rook};
        \node at (2.5,-.5) {$\color{Red}\times$};
        \node at (1.5,-1.5) {$\color{Red}\times$};
        \node at (1.5,-2.5) {$\color{Red}\times$};
    \end{tikzpicture}
    \caption{\centering $\Pi(\mF,4,4)$.}
    \label{fig:reductext2}
    \end{minipage}
    \begin{minipage}{0.32\textwidth}
    \centering
    \begin{tikzpicture}[scale=0.6, line width=1pt]
        \draw (1,0) grid (6,-5);
         \foreach \n in {0,...,0} {\draw[draw=black, fill =Goldenrod!55] (1,-\n) rectangle (2,-1-\n);}
         \foreach \n in {0,...,0} {\draw[draw=black, fill =Goldenrod!55] (2,-\n) rectangle (3,-1-\n);}
    \end{tikzpicture}
    \caption{\centering $\Pi(\mF,P)$.}
    \label{fig:reductext}
    \end{minipage}
    \end{figure}
\end{example}

The following result shows that diagonal equivalence is invariant under reduction.

\begin{lemma} \label{lem:punctdiageq}
Let $\mF$ and $\mF'$ be diagonally equivalent Ferrers diagrams in $\N\times\N$. Then $\Pi(\mF,i,j)$ and $\Pi(\mF',i,j)$ are diagonally equivalent for all $i \in \N$ and $j \in \{1, \dots, |\diag_i\cap\mF|\}$.
\end{lemma}
\begin{proof}
We first prove the following claim.

\begin{claim}
Let $(\diag_i\cap\mF)_j$ be the greatest element of $\mF$ with respect to the lexicographic order, that is $(\diag_i\cap\mF)_j=\max_\lex(\mF)$, and suppose that $(\diag_i\cap\mF)_j=(a,b)$, for some $a,b\in\N$. Then $\Pi(\mF,i,j)$ is the sub-board of $\N\times \N$ obtained by deleting the cells in the set $\{(1,1),\ldots,(1,b),\ldots,(a,b)\}$ and aligning the remaining cells to the top and then to the left.
\end{claim}
\begin{clproof}
Recall that, by Definition~\ref{def:reduction} and Notation~\ref{not:atk}, we have that $(\diag_i\cap\mF)_j$ is the Ferrers sub-diagram of $\N\times \N$ obtained by deleting the set  $\atk(\{(\diag_i\cap\mF)_j\})$, that is $\{(1,b),\ldots,(a,b),(a,b-1),\ldots,(a,1)\}$, since $(\diag_i\cap\mF)_j=\max_\lex(\mF)$. One can check that $\atk(\{(\diag_i\cap\mF)_j\})$ and $\{(1,1),\ldots,(1,b),\ldots,(a,b)\}$ contain the same number of cells in the rows and columns of $\mF$ and the statement easily follows.
\end{clproof}

We also note the following.
\begin{itemize}
    \item[(i)] Let $\mF'$ be obtained by deleting the set  $\{(\diag_s\cap\mF)_t:s\in\N, t\in\{1,\ldots,|\diag_s\cap\mF|\},(i,j)<_\lex (s,t)\}$ from $\mF$. Then $(\diag_i\cap\mF)_j=\max_\lex(\mF')$.
    \item[(ii)] For any $s\in\{1,\ldots,i\}$, the set $\{(1,1),\ldots,(1,b),\ldots,(a,b)\}$ intersects $\diag_s(\mF)$ in exactly one cell.
    \item[(iii)] The Ferrers diagram $\Pi(\mF,i,j)$ has $i-1$ diagonals.
\end{itemize}

As a consequence,  $|\diag_s\cap\Pi(\mF,i,j)|=|\diag_{s+1}\cap\mF|-1$ for any $s\in\{1,\ldots,i-1\}$, and $|\diag_s\cap\Pi(\mF,i,j)|=0$ for $s\geq i$. The statement now follows from the fact that $\mF$ and~$\mF'$ are diagonally equivalent.
\end{proof}

\section{Equivalence of Ferrers Diagrams and Applications}\label{sec:equiv}\label{sec:3}

This section contains the main results of the paper, showing that
the $q$-rook polynomials of a Ferrers diagram $\mF$ are fully determined by the cardinalities of its diagonals. Note that these cardinalities are far from being enough to determine the diagram~$\mF$ itself.

\begin{theorem}
\label{thm:bijection}
Let $\mF, \mF' \subseteq \N\times \N$ be diagonally equivalent Ferrers diagrams in $\N\times\N$. For~$r\in\N_0$ and $\ell\in\N$, we have  
\begin{align*}
    |\{P \in \NAR(\mF, r) : \inv(P) = \ell \}| = |\{P  \in \NAR(\mF', r) : \inv(P ) = \ell \}|.
\end{align*} 
In particular, $\mF$ and $\mF'$ have the same $r$-th $q$-rook polynomial for all $r\in\N_0$.
\end{theorem}

\begin{proof}
We will show that there is a one-to-one correspondence between placements $P \in \NAR(\mF, r)$ with $\inv(P)=\ell$ on the one hand, and placements $P' \in \NAR(\mF', r)$ with \smash{$\inv(P')=\ell$} on the other hand, under the assumption that $\Pi(\mF,P)$ and $\Pi(\mF',P')$ are diagonally equivalent.  
We proceed by induction on $r$. For $r=1$ we have 
\begin{align*}
    |\{P \in \NAR(\mF, 1) : \inv(P) = \ell \}| = |\diag_\ell\cap\mF| = |\{P' \in \NAR(\mF', 1) : \inv(P') = \ell \}|,
\end{align*}
where we used that $|\diag_\ell\cap\mF| = |\diag_\ell\cap\mF'|$.
More in detail, the map $(\diag_\ell\cap\mF)_s \mapsto (\diag_\ell\cap\mF')_s$ for $s \in \{1, \dots, |\diag_\ell\cap\mF|\}$ is a bijection between the elements of $\diag_\ell\cap\mF$ and the elements of $\diag_\ell\cap\mF'$ and by Lemma~\ref{lem:punctdiageq}, $\Pi(\mF,\ell,s)$ and $\Pi(\mF',\ell,s)$ are diagonally equivalent.

Now assume that the statement of the theorem holds true for all the placements of $r-1$ rooks.  Let $P=\{(\diag_{i_1}\cap\mF)_{j_1},\ldots,(\diag_{i_r}\cap\mF)_{j_r}\}\in\NAR(\mF,r)$ with $\inv(P) = \ell$ and where $(i_1,j_1) \le_{\lex} \dots \le_{\lex} (i_r,j_r)$. Consider $\overline{P}=\{(\diag_{i_2}\cap\mF)_{j_2},\ldots,(\diag_{i_r}\cap\mF)_{j_r}\}$. Then $|\overline{P}| = r-1$ and by assumption there exists a rook placement $\overline{P}' \subseteq \mF'$ with $|\overline{P}'| = r-1$, $\inv(P') = \inv(\overline{P}')$, and such that $\Pi(\mF,\overline{P})$ and $\Pi(\mF',\overline{P}')$ are again diagonally equivalent. Since no rook in $\overline{P}$ attacks $(\diag_{i_1}\cap\mF)_{j_1}$, there is a cell in $\Pi(\mF,\overline{P})$ corresponding to $(\diag_{i_1}\cap\mF)_{j_1}$ (according to the reduction) which we denote by $(\diag_{a}\cap\mF)_b \in \Pi(\mF,\overline{P})$. We let $\overline{P} = \overline{P}' \cup (\diag_{a}\cap\mF')_b$, and it is easy to see that $\overline{P}$ satisfies all the needed properties, concluding the proof.
\end{proof}

As an easy consequence, we obtain the following result. It implies that
the number of matrices of rank $r$ supported on $\mF$ only depends on the cardinalities of the diagonals of~$\mF$.
We are not able to provide an intuitive or easy explanation for this fact. The result also shows that two Ferrers diagrams have the same (classical) rook numbers
if and only if their diagonals have the same cardinalities.

\begin{corollary}\label{cor:surprise}
    Let $\mF,\mF' \subseteq \{1,\dots,n\} \times \{1,\dots,m\}$ be Ferrers diagrams. The following are equivalent. 
    \begin{enumerate}
        \item $\mF$ and $\mF'$ are diagonally equivalent.
        \item $R_1(\mF;q)=R_1(\mF';q)$.
        \item $R_r(\mF;q)=R_r(\mF';q)$ for all $r\in\N_0$.
        \item $|W_1(\mat[\mF])|=|W_1(\mat[\mF'])|$.
        \item $|W_r(\mat[\mF])|=|W_r(\mat[\mF'])|$ for all $r\in\N_0$.
        \item $\mF$ and $\mF'$ have the same (classical) rook numbers.
    \end{enumerate}
\end{corollary}
\begin{proof}
Define $d_i=|\triangle_i \cap \mF|$ for all $i \in \N$. By Remark~\ref{rem:power} we have
\begin{align*}
    R_1(\mF;q) = \sum_{P \in \mF}q^{\inv(\mF,P)} = \displaystyle\sum_{i=1}^\ell d_iq^{|\mF|-i},
\end{align*}
where $\ell = \max\{i : d_i \ne 0\}$. This shows that 2. implies 1.
The fact that 6. is equivalent to~3. was shown in~\cite[page 257]{GaRe86}.
We omit the proofs for the other implications, as they can be checked easily or obtained from Theorem~\ref{thm:bijection}.
\end{proof}

\begin{remark}\label{rem:ferrseq}
Fix $r,\ell\in\N$ and $d_1,\ldots,d_\ell\in\N_0$. Then there exists a Ferrers diagram $\mF$ with
$|\diag_i\cap\mF|=d_i$ for all $i\in\{1,\ldots,\ell\}$ if and only if $d_1=\cdots=d_\ell=0$ or there exists an integer $r \in \{1, \dots, \ell\}$ such that $d_1=1, d_2=2, \dots, d_r=r$ and $d_r \ge d_{r+1} \ge d_\ell$. 
For such a sequence and for a Ferrers diagram $\mF$ with $|\diag_i\cap\mF|=d_i$ for all $i$, by Proposition~\ref{prop:dFdiag} we have that $r=\max\{d_1,\dots,d_\ell\}=\partial(\mF)$. 
We write $\partial(d_1,\dots,d_\ell)$ instead of $\partial(\mF)$. 
\end{remark}

The previous remark motivates the following concepts.

\begin{definition} \label{def:fdseq}
Let $\ell\in\N$ and let $d_1,\ldots,d_\ell\in\N$. We say that $(d_1,\ldots,d_\ell)$ is a \textbf{Ferrers sequence} if $(d_1,\ldots,d_\ell)=(0,\ldots,0)$ or there exists $r \in \{1, \dots, \ell\}$ such that $d_1=1, d_2=2, \dots, d_r=r$ and   $d_r \ge \dots \ge d_\ell$. For a Ferrers sequence $(d_1,\ldots,d_\ell)$
and $r \in \N$, we denote by $R_q(d_1,\ldots,d_\ell;q)$ the $r$-th $q$-rook polynomial of a Ferrers diagram $\mF$ such that $\diag_i\cap\mF=d_i$ for all $i\in\{1,\ldots,\ell\}$. We denote by $\mS_\ell$ the set of Ferrers sequences of length $\ell$.
\end{definition}

For example, the diagrams in Figures~\ref{fig:equiv1}, \ref{fig:equiv2} and~\ref{fig:equiv3} are diagonally equivalent and all correspond to the Ferrers sequence $(1,2,3,3,2)$.
	\begin{figure}[h]
    \begin{minipage}{0.32\textwidth}
    \centering
    \begin{tikzpicture}[scale=0.6, line width=1pt]
        \draw (0,0) grid (5,-5);
         \foreach \n in {0,...,4} {\draw[draw=black, fill =Goldenrod!55] (0,-\n) rectangle (1,-1-\n);}
         \foreach \n in {0,...,1} {\draw[draw=black, fill =Goldenrod!55] (1,-\n) rectangle (2,-1-\n);}
         \foreach \n in {0,...,1} {\draw[draw=black, fill =Goldenrod!55] (2,-\n) rectangle (3,-1-\n);}
         \foreach \n in {0,...,0} {\draw[draw=black, fill =Goldenrod!55] (3,-\n) rectangle (4,-1-\n);}
          \foreach \n in {0,...,0} {\draw[draw=black, fill =Goldenrod!55] (4,-\n) rectangle (5,-1-\n);}
    \end{tikzpicture}
    \caption{$[5,2,2,1,1]$}
    \label{fig:equiv1}
    \end{minipage}
    \begin{minipage}{0.32\textwidth}
    \centering
    \begin{tikzpicture}[scale=0.6, line width=1pt]
        \draw (0,0) grid (5,-5);
         \foreach \n in {0,...,3} {\draw[draw=black, fill =Goldenrod!55] (0,-\n) rectangle (1,-1-\n);}
         \foreach \n in {0,...,3} {\draw[draw=black, fill =Goldenrod!55] (1,-\n) rectangle (2,-1-\n);}
         \foreach \n in {0,...,2} {\draw[draw=black, fill =Goldenrod!55] (2,-\n) rectangle (3,-1-\n);}
    \end{tikzpicture}
    \caption{$[4,4,3]$}
    \label{fig:equiv2}
    \end{minipage}
        \begin{minipage}{0.32\textwidth}
    \centering
    \begin{tikzpicture}[scale=0.6, line width=1pt]
        \draw (0,0) grid (5,-5);
         \foreach \n in {0,...,2} {\draw[draw=black, fill =Goldenrod!55] (0,-\n) rectangle (1,-1-\n);}
         \foreach \n in {0,...,2} {\draw[draw=black, fill =Goldenrod!55] (1,-\n) rectangle (2,-1-\n);}
         \foreach \n in {0,...,1} {\draw[draw=black, fill =Goldenrod!55] (2,-\n) rectangle (3,-1-\n);}
         \foreach \n in {0,...,1} {\draw[draw=black, fill =Goldenrod!55] (3,-\n) rectangle (4,-1-\n);}
          \foreach \n in {0,...,0} {\draw[draw=black, fill =Goldenrod!55] (4,-\n) rectangle (5,-1-\n);}
    \end{tikzpicture}
    \caption{$[3,3,2,2,1]$}
    \label{fig:equiv3}
    \end{minipage}
	\end{figure}

In the next theorem we give a recursive formula for the $q$-rook polynomials of a Ferrers diagram only taking into account its diagonals. The result is the diagonal-analogue
of~\cite[Theorem~1.1]{GaRe86}.

\begin{theorem}
\label{thm:recursive}
    Let $r\in\Z$, $\ell\in\N$ and let $(d_1,\ldots,d_\ell)\in \mS_\ell$ be a Ferrers sequence. We have 
    \begin{equation*}
        R_r(d_1,\ldots,d_\ell;q)= R_r(d_1,\ldots,d_{\ell-1};q)q^{d_\ell}+\sum_{j=1}^{d_\ell}R_{r-1}(d_2-1,\ldots,d_{\ell-1}-1,j-1;q)q^{d_\ell-j}
    \end{equation*}
    if $1\le r\le \partial(d_1,\dots,d_\ell)$, and
    \begin{equation*}
        R_r(d_1,\ldots,d_\ell;q)=
        \begin{cases}
            q^{|\mF|} & \textup{ if } r=0,\\
            0 & \textup{ if }  r> \partial(d_1,\dots,d_\ell)  \textup{ or } r<0.
        \end{cases}
    \end{equation*}
\end{theorem}
\begin{proof}
The cases $r=0$, $r > \partial(d_1,\dots,d_\ell)$ and $r<0$ follow immediately from the definition of $q$-rook polynomial. We assume $1 \le r\le \partial(d_1,\dots,d_\ell)$ in the remainder of the proof. By Theorem~\ref{thm:bijection}, it suffices to prove the statement for a Ferrers diagram $\mF$ with $|\diag_i\cap\mF|=d_i$ for $i\in\{1,\ldots,\ell\}$ and $|\diag_i\cap\mF|=0$ otherwise. By the definition of puncturing and reduction (see Definition~\ref{def:reduction}) and Lemma~\ref{lem:interreduc}, we have
\begin{align*}
     R_r(\mF;q)&=\sum_{P\in\NAR(\mF)}q^{\inv(P)} \\
     &= R_r(\widehat\mF;q)q^{d_\ell}+\sum_{j=1}^{d_\ell}\Pi(\mF,\ell,j)q^{|\{(\diag_\ell\cap\mF)_s:s\in\{j+1,\ldots,d_\ell\}\}|}\\
     &=R_r(\widehat\mF;q)q^{d_\ell}+\sum_{j=1}^{d_\ell}\Pi(\mF,\ell,j)q^{d_\ell-j}.
\end{align*}
Moreover, the following hold.
\begin{itemize}
    \item[(i)] By definition of puncturing,  $\widehat\mF$ is such that $\diag_i\cap\widehat\mF=\diag_i\cap\mF=d_i$ if $i\in\{1,\ldots,\ell-1\}$ and $|\diag_i\cap\widehat\mF|=0$ otherwise.
    
    \item[(ii)] By the argument used in  Lemma~\ref{lem:punctdiageq}, the diagram $\Pi(\mF,\ell,j)$ satisfies
    \begin{align*}
    \begin{cases}
        |\diag_i\cap\Pi(\mF,\ell,j)|=|\diag_i\cap\mF|-1=d_i-1 \quad &\textnormal{if $i\in\{1,\ldots,\ell-1\}$,} \\
        |\diag_i\cap\Pi(\mF,\ell,j)|=j-1 \quad &\textnormal{if $i= \ell$,} \\
        |\diag_i\cap\Pi(\mF,\ell,j)|=0 \quad &\textnormal{if $i >\ell$.}
    \end{cases}
    \end{align*}
\end{itemize}
This concludes the proof.
\end{proof}

We have seen that the $q$-rook polynomials of a Ferrers diagram are fully determined by the cardinalities of its diagonals. In particular,
diagonally equivalent Ferrers diagrams share the same $q$-rook polynomials.
The next definition selects the ``best'' representative.

\begin{definition} \label{def:canon}
Let $\mF$ be a non-empty Ferrers diagram. The \textbf{canonical form} $\overline{\mF}$ of $\mF$ is the Ferrers diagram uniquely defined by
\begin{equation*}
        \diag_i\cap\overline{\mF}=\{(s,i-s+1): s\in\{1,\ldots,|\diag_i\cap\mF|\}\} \quad \textnormal{for all $i\in\N$.}
\end{equation*}
We set the canonical form of $\mF=[]$ to be $\overline{\mF}=[]$.
\end{definition}

Note that the canonical form is obtained by aligning all the diagonals to the top border of $\N \times \N$, as explained in Example~\ref{ex:canonical} and the related figures.

\begin{example}\label{ex:canonical}
    Consider the Ferrers diagram $\mF=[5,3,3,2]$. The canonical form $\overline{\mF}$ is the Ferrers diagram $[4,3,3,2,1]$. Figures~\ref{fig:F} and~\ref{fig:Fbar} give a visualization of these diagrams.

    \begin{figure}[h]
    \begin{minipage}{0.49\textwidth}
        \centering
        \begin{tikzpicture}[scale=0.6, line width=1pt]
        \draw (0,0) grid (5,-5);
         \foreach \n in {0,...,4} {\draw[draw=black, fill =Goldenrod!55] (0,-\n) rectangle (1,-1-\n);}
         \foreach \n in {0,...,2} {\draw[draw=black, fill =Goldenrod!55] (1,-\n) rectangle (2,-1-\n);}
         \foreach \n in {0,...,2} {\draw[draw=black, fill =Goldenrod!55] (2,-\n) rectangle (3,-1-\n);}
         \foreach \n in {0,...,1} {\draw[draw=black, fill =Goldenrod!55] (3,-\n) rectangle (4,-1-\n);}
    \end{tikzpicture}
    \captionof{figure}{$\mF$}
    \label{fig:F}
    \end{minipage}\hfill
    \begin{minipage}{0.49\textwidth}
        \centering
        \begin{tikzpicture}[scale=0.6, line width=1pt]
        \draw (0,0) grid (5,-5);
         \foreach \n in {0,...,3} {\draw[draw=black, fill =Goldenrod!55] (0,-\n) rectangle (1,-1-\n);}
         \foreach \n in {0,...,2} {\draw[draw=black, fill =Goldenrod!55] (1,-\n) rectangle (2,-1-\n);}
         \foreach \n in {0,...,2} {\draw[draw=black, fill =Goldenrod!55] (2,-\n) rectangle (3,-1-\n);}
         \foreach \n in {0,...,1} {\draw[draw=black, fill =Goldenrod!55] (3,-\n) rectangle (4,-1-\n);}
        \foreach \n in {0,...,0} {\draw[draw=black, fill =Goldenrod!55] (4,-\n) rectangle (5,-1-\n);}
    \end{tikzpicture}
    \captionof{figure}{$\overline{\mF}$}
    \label{fig:Fbar}
    \end{minipage}
    \end{figure}   
\end{example}

It turns out that a Ferrers diagram is in canonical form if and only if it is \textit{initially convex} in the sense of~\cite[Definition~4.4]{neri2023proof}.
Despite being curiously the same, the two notions appear in different contexts and for different purposes. 

We recall that the \emph{trailing degree} of a polynomial $f(q)=\sum_{i} a_iq^i \in \Z[q]$ is the smallest~$i$ with $a_i \neq 0$, where the zero polynomial has trailing degree $-\infty$. From now on, we denote the trailing degree of a polynomial $f(q)\in \Z[q]$ by $\tau(f(q))$. The next lemma shows that if a diagram is in canonical form, then the trailing degree of its $q$-rook polynomials can be easily computed. We omit the proof since it directly follows from the definition of canonical form.

\begin{lemma}
Let $\mF \subseteq \N\times\N$ be a Ferrers diagram and let $r\in\N$. If $\mF$ is in canonical form then a placement $P$ of $r$ non-attacking rooks with $\inv(P) = \tau(R_r(\mF;q))$, should it exist, is obtained by placing a rook in the rightmost cell of each of the $r$ topmost rows of $\mF$.  
\end{lemma}

This immediately implies the following result from~\cite{gruica2022rook}.

\begin{theorem}[{\cite[Theorem~2.15]{gruica2022rook}}]
\label{thm:trail}
Let $\mF$ be a Ferrers diagram in $\N\times\N$. For any $1 \le r \le \partial(\mF)$ we have
    \begin{equation*}
        \tau(R_r(\mF;q))=\sum_{i\in\N}\max\{0,|\diag_i\cap\mF-r|\}.
    \end{equation*}
\end{theorem}

\begin{remark}
In~\cite{gruica2022rook}, it was also shown that the trailing degree of a $q$-rook polynomial is closely related with an open conjecture about the largest dimension of a space of matrices supported on a Ferrers diagram in which all nonzero matrices have rank at least~$d$; see~\cite{etzion2009error} for the problem and~\cite[Theorem 3.6]{gruica2022rook}. Since the number of rank $d$ matrices supported on a Ferrers diagram is the same 
for all diagrams in the same equivalence class, it is natural to ask if the same property holds for spaces of matrices as well. The answer is in general negative.
\end{remark}

\section{Symmetric Ferrers Diagrams} \label{sec:symandalt}\label{sec:4}

In this section 
we develop the theory of Ferrers diagrams, with a focus on their diagonals,
for alternating and symmetric matrices.
To this end, 
we introduce new definitions of $q$-rook polynomials of \textit{symmetric} Ferrers diagrams. These will allow us to establish the analogues of the results of the previous sections for the alternating and symmetric case.

Recall that a matrix $M \in \operatorname{Mat}_q^{n \times n}$ is \textbf{symmetric} if $M_{ij} = M_{ji}$ for all $i,j \in \{1, \dots, n\}$ and 
\textbf{alternating} if $M_{ii}=0$ and $M_{ij} = -M_{ji}$ for all $i,j \in \{1, \dots, n\}$. 
We denote the space of symmetric and alternating matrices by $\matsym$ and
$\matalt$, respectively.
We denote by $\Xi=\{(i,i) \in \N \times \N\}$ the \textbf{principal diagonal} of $\N \times \N$.

We already observed that a matrix of rank $r$ can be associated to a placement of non-attacking rooks in $\N\times\N$; see also~\cite{haglund1998q}. If we restrict to symmetric and alternating matrices, such a placement can only be of a certain form. This motivates the following definition.

\begin{definition}
    A Ferrers diagram $\mF$ is said to be \textbf{symmetric} if $(i,j) \in \mF$ implies $(j,i) \in \mF$.
    For $r\in\N_0$, we let
\begin{equation*}
        \NAR^\alt(\mF,r)=\{P \in \NAR(\mF,r) : (i,j) \in P \mbox{ implies }  i \ne j \mbox{ and } (j,i) \in P\}.
\end{equation*}
Clearly, if $r$ is odd then $\NAR^\alt(\mF,r) = \emptyset$. For any $t,s\in\N_0$, we let
\begin{multline*}
    \NAR^\sym(\mF,t,s) \\ =\{P\in\NAR(\mF,2t+s): P=T\cup S \textup{ with } T\in\NAR^\alt(\mF,2t), \, S\subseteq \Xi, \textup{ and } |S|=s\}.
\end{multline*}
We call placements in $\NAR^\alt(\mF,r)$ \textbf{alternating} and placements in ${\NAR^\sym(\mF,t,s)}$ \textbf{symmetric}.
\end{definition}

\begin{example}
    Let $\mF=[5,3,2,1,1]$ be a Ferrers diagram and consider the placements of non-attacking rooks defined as follows:
    \begin{align*}
        S:=&\{(\diag_1\cap\mF)_1,(\diag_4\cap\mF)_2,(\diag_4\cap\mF)_3\},\\
        T:=&\{(\diag_4\cap\mF)_2,(\diag_4\cap\mF)_3,(\diag_5\cap\mF)_1,(\diag_5\cap\mF)_2\}.
    \end{align*}
    As one can see from Figures~\ref{fig:sym} and~\ref{fig:alt}, $\mF$ is symmetric and $S$ and $T$ are a symmetric and an alternating placement, respectively. 
    \begin{figure}[h]
    \begin{minipage}{0.49\textwidth}
        \centering
        \begin{tikzpicture}[scale=0.6, line width=1pt]
        \draw (0,0) grid (5,-5);
         \foreach \n in {0,...,4} {\draw[draw=black, fill =Goldenrod!55] (0,-\n) rectangle (1,-1-\n);}
         \foreach \n in {0,...,2} {\draw[draw=black, fill =Goldenrod!55] (1,-\n) rectangle (2,-1-\n);}
         \foreach \n in {0,...,1} {\draw[draw=black, fill =Goldenrod!55] (2,-\n) rectangle (3,-1-\n);}
         \foreach \n in {0,...,0} {\draw[draw=black, fill =Goldenrod!55] (3,-\n) rectangle (4,-1-\n);}
        \foreach \n in {0,...,0} {\draw[draw=black, fill =Goldenrod!55] (4,-\n) rectangle (5,-1-\n);}
        \node at (.5,-.5) {\color{black}\rook};
        \node at (1.5,-2.5) {\color{black}\rook};
        \node at (2.5,-1.5) {\color{black}\rook};
    \end{tikzpicture}
    \captionof{figure}{$\mF$ with the placement $S$.}
    \label{fig:sym}
    \end{minipage}\hfill
    \begin{minipage}{0.49\textwidth}
        \centering
        \begin{tikzpicture}[scale=0.6, line width=1pt]
        \draw (0,0) grid (5,-5);
         \foreach \n in {0,...,4} {\draw[draw=black, fill =Goldenrod!55] (0,-\n) rectangle (1,-1-\n);}
         \foreach \n in {0,...,2} {\draw[draw=black, fill =Goldenrod!55] (1,-\n) rectangle (2,-1-\n);}
         \foreach \n in {0,...,1} {\draw[draw=black, fill =Goldenrod!55] (2,-\n) rectangle (3,-1-\n);}
         \foreach \n in {0,...,0} {\draw[draw=black, fill =Goldenrod!55] (3,-\n) rectangle (4,-1-\n);}
        \foreach \n in {0,...,0} {\draw[draw=black, fill =Goldenrod!55] (4,-\n) rectangle (5,-1-\n);}
        \node at (.5,-4.5) {\color{black}\rook};
         \node at (4.5,-.5) {\color{black}\rook};
        \node at (1.5,-2.5) {\color{black}\rook};
        \node at (2.5,-1.5) {\color{black}\rook};
    \end{tikzpicture}
    \captionof{figure}{$\mF$ with the placement $T$.}
    \label{fig:alt}
    \end{minipage}
    \end{figure}
\end{example}

Ferrers diagrams in connection with alternating matrices have been considered before by Haglund and Remmel in \cite{haglund2001rook}. There the authors consider different $q$-rook polynomials associated to a special finite sub-board of $\{(i,j)\in \N\times\N: j>i\}$, called a \textit{shifted} Ferrers diagram, which is a finite subset $\mG\subseteq\{(i,j)\in \{1, \dots, n\} \times \{1, \dots, n\}: j>i\}$ with $(1,2)\in\mG$ and the property that if $(i,j)\in\mG$ then $(s,t)\in\mG$ for all $s\in\{1,\ldots,i\}$ and $t\in\{1,\ldots,j\}$ with~$t > s$. In this paper we take a different approach, which can however be easily related to the one of~\cite{haglund2001rook}. Note that, to the best of our knowledge, the theory of symmetric matrices has not been considered before.

We start by relating the degree of a Ferrers diagram with the ranks of symmetric and alternating matrices supported on it.

\begin{lemma}\label{lem:partialF}
    Let $\mF \subseteq \{1, \ldots, n\} \times \{1, \ldots, n\}$ be a symmetric Ferrers diagram. The following hold.
    \begin{enumerate}
        \item $\partial(\mF)=\max\{\rk(M) : M \in \matsym[\mF]\}$.
        \item If $\partial(\mF)$ is even, then $\partial(\mF)=\max\{\rk(M) : M \in \matalt[\mF]\}$.
    \end{enumerate}
\end{lemma}
\begin{proof}
    Recall that by Proposition~\ref{prop:partialFmaxrk}, Proposition~\ref{prop:dFdiag}, and \cite[Claim~A]{gruica2022rook}, we have
    \begin{equation*}
        \partial(\mF)=\max\{\rk(M) : M \in \mat[\mF]\}=\max\{|\diag_i\cap\mF|:i\in\N\}=|\diag_{\partial(\mF)}\cap\mF|.
    \end{equation*}
     Observe that $\partial(\mF)\geq\max\{\rk(M) : M \in \matsym[\mF]\}$ since $\matsym[\mF]\subseteq\mat[\mF]$. On the other hand, define $M\in\mat$ by $M_{ij}=1$ if $(i,j)\in\diag_{\partial(\mF)}\cap\mF$ and $M_{ij}=0$ otherwise. Then $M\in\matsym[\mF]$ and $\rk(M)=\partial(\mF)$, which implies the first part of the statement. The second part can be shown similarly.
\end{proof}

We introduce the following $q$-rook polynomials for alternating placements of non-attacking rooks.

\begin{definition}
Let $\mF$ be a symmetric Ferrers diagram and let $r \in \Z$. The \textbf{alternating $r$-th $q$-rook polynomial} of $\mF$ is
\begin{align*}
R_r^\alt(\mF;q) :=\sum_{P \in \NAR^\alt(\mF,r)} q^{\inv(\mF,P)+r/2-|\Xi\cap\mF|},
\end{align*}
where we set $R_r^\alt(\mF;q)=0$ if $r<0$ or $\NAR^\alt(\mF,r)=\emptyset$, and $R_0(\mF;q)=q^{|\mF|-|\Xi\cap\mF|}$.
\end{definition}

Let $\mF\subseteq\{1, \dots, n\} \times \{1, \dots, n\}$ be a symmetric Ferrers diagram and $r\in\N$. It is easy to see that there exists a placement of non-attacking rooks of size $r$ on $\mF$ if and only if there exists a matrix of rank $r$ in $\matalt[\mF]$. Moreover, the following result holds. Since the proof is similar to the one of~\cite[Theorem 18]{haglund2001rook}, we omit it here.

\begin{theorem}\label{thm:altqrooken}
    Let $\mF\subseteq\{1,\ldots,n\}\times\{1,\ldots,n\}$ be a symmetric Ferrers diagram and let~$r\in\Z$. Then 
    \begin{equation*}
        W_r(\matalt[\mF]) = (q-1)^{r/2} q^{(|\mF|-|\Xi\cap\mF|-r)/2} R_{r}^\alt(\mF;q^{-1/2}).
    \end{equation*}
\end{theorem}

Symmetric matrices are somewhat more complicated than alternating matrices.
We introduce a $q$-rook polynomial for symmetric placements of non-attacking rooks, with the goal
of connecting it to the problem of counting symmetric matrices of given rank supported on a symmetric Ferrers diagram.
In particular, we wish to derive 
a recursive formula similar to the one in Theorem~\ref{thm:recursive}. To this end, we find it necessary to introduce a refinement of $q$-rook polynomials that takes into account the number of non-attacking rooks in a symmetric placement lying on the main diagonal.

\begin{definition}
Let $\mF\subseteq\{1, \dots, n\} \times \{1, \dots, n\}$ be a symmetric Ferrers diagram and let $t,s\in\Z$. The \textbf{symmetric $(t,s)$-th $q$-rook polynomial} associated to $\mF$ is
\begin{equation*}
    R_{t,s}^\sym(\mF;q)= \sum_{P \in \NAR^\sym(\mF,t,s)} q^{\inv(\mF,P)},
\end{equation*}
where we set $R_{t,s}^\sym(\mF;q)=0$ if $t<0$, $s<0$ or $\NAR^\sym(\mF,t,s) = \emptyset$, and $R_{0,0}^{\sym}(\mF;q)=q^{|\mF|}$. 
\end{definition}

    It is not hard to see that, for any $t\in\N_0$, $\NAR^\sym(\mF,t,0)=\NAR^\alt(\mF,2t)$ and therefore $ R_{t,0}^\sym(\mF;q)=q^{|\Xi\cap\mF|-t} R_{2t}^\alt(\mF;q)$.
Furthermore,
there exists a symmetric placement of non-attacking rooks
of size $2t+s$ on a symmetric Ferrers diagram $\mF$ if and only if there exists a matrix of rank $2t+s$ in $\matsym[\mF]$.

The next result is the analogue of~\cite[Theorem 18]{haglund2001rook} for symmetric matrices. It gives an enumeration formula for the number of symmetric matrices supported on a Ferrers diagram using their associated $q$-rook polynomials.

\begin{theorem} \label{thm:symcount}
Let $\mF \subseteq \{1, \dots, n\} \times \{1, \dots, n\}$ be a symmetric Ferrers diagram and let $r \in \Z$. Then
\begin{align} \label{eq:thmhagskew}
    W_r(\matsym[\mF]) = \sum_{\substack{t,s \ge 0 \\ 2t+s=r} } (q-1)^{t+s} q^{(|\mF|-t-s)/2} R_{t,s}^\sym(\mF;q^{-1/2}).
\end{align}
\end{theorem}
\begin{proof}
The result is trivial for $r<0$ and we assume $r\in\N_0$ in the remainder of the proof. Let $M \in \matsym[\mF]$ be a symmetric matrix of rank $r$. We denote by $M_i$ and $M^i$ the $i$-th row and the $i$-th column of $M$ respectively, where $i\in\{1,\ldots,n\}$. We perform a modified Gaussian elimination on $M$ which preserves the symmetry of $M$. This procedure consists of the following steps:
\begin{enumerate}
    \item Find the lowest nonzero entry $(i,j)$ of $M$ in the right-most column.
    \item Delete entries above $(i,j)$ by adding suitable multiple of $M_i$ to the corresponding row. Formally, we perform the elementary row operation $M_s\rightarrow M_{ij}M_s-M_{sj}M_i$ for all $s\in\{1,\ldots,i-1\}$.
    \item Delete entries to the left of $(i,j)$ by adding suitable multiple of $M^j$ to the corresponding columns. Formally, we perform the elementary column operation $M^s\rightarrow M_{ij}M^s-M_{is}M^j$ for all $s\in\{1,\ldots,j-1\}$.
    \item If $i \ne j$, then perform the same procedure to entry $(j,i)$ as well.
\end{enumerate}
It is not difficult to check that the resulting matrix is symmetric. We iterate this procedure until there are only $r$ nonzero entries left. These entries form a placement of $r$ non-attacking rooks on $\mF$. Note that any matrix of rank $r$ in $\matsym[\mF]$ generates a unique symmetric placement of non-attacking rooks $P \in \NAR(\mF,r)$ with $|\{(i,j) \in P : i \ne j\}|=2t$ and $|\{(i,j) \in P : i = j\}|=s$ for some $t,s \ge 0$ with $2t+s=r$. In particular, for such a rook placement $P$, we have $(q-1)$ choices for each rook in $\{(i,j) \in P : i < j\}$ and $(q-1)$ choices for each rook in $\{(i,j) \in P : i = j\}$, which gives a total of $(q-1)^{t+s}$ choices for the cells containing a rook. The number of the cells that are deleted by rooks
when computing the $\inv$ statistics is $(|\mF|-\inv(\mF,P))/2+(t+s)/2-t-s$, taking into account that the matrix is symmetric and that we do not count the cells containing a rook themselves. This gives a total of $(q-1)^{t+s} q^{(|\mF|-t-s-\inv(\mF,P))/2}$ matrices with $P$ as underlying rook placement. The statement follows by summing over all possible symmetric placements of non-attacking rooks. 
\end{proof}

\begin{remark}
Even if it is not immediate to see from the formula in Theorem~\ref{thm:symcount}, the number of symmetric matrices of a given rank supported on a symmetric Ferrers diagram $\mF \subseteq \{1,\dots,n\} \times \{1,\dots,n\}$ is a polynomial $q$ (where $q$ is the underlying field size). To see this, for $r \in \N_0$ we can rewrite~\eqref{eq:thmhagskew} as 
\begin{align*}
    W_r(\matsym[\mF]) &= \sum_{\substack{t,s \ge 0 \\ 2t+s=r} } (q-1)^{t+s} q^{(|\mF|-t-s)/2} \sum_{P \in \NAR^\sym(\mF,t,s)} q^{-\inv(\mF,P)/2} \\
    &= \sum_{\substack{t,s \ge 0 \\ 2t+s=r} } (q-1)^{t+s} \sum_{P \in \NAR^\sym(\mF,t,s)} q^{(|\mF|-\inv(\mF,P))/2+(t+s)/2-t-s}.
\end{align*}
In order for the above expression to be a polynomial in $q$, it is enough to show that for any $t,s \ge 0$ with $2t+s=r$ and any $P \in \NAR^\sym(\mF,t,s)$, the number $$(|\mF|-\inv(\mF,P))/2+(t+s)/2-t-s$$ is a non-negative integer. This follows from the fact that $(|\mF|-\inv(\mF,P))/2+(t+s)/2  = \{(i,j) \in \atk(P) : i \le j \}$, which is not hard to see.

Note that also the number of alternating matrices of a certain rank supported on a Ferrers diagram is a polynomial in $q$. This can be explained in a similar manner, using the enumeration formula of Theorem~\ref{thm:altqrooken}.
\end{remark}

In the remainder of this section, we derive recursive formulas for the alternating and symmetric $q$-rook polynomials. These can be seen as the analogues of Theorem~\ref{thm:recursive} for symmetric Ferrers diagrams. We start by introducing a new version of the reduction operation. This is motivated by the fact that the reduction of a symmetric Ferrers diagram according to Definition~\ref{def:reduction} is, in general, not symmetric.

\begin{definition}\label{def:symred}
    Let $\mF$ be a symmetric Ferrers diagram, $i\in\N$ and $j\in\{1,\ldots,i\}$. The \textbf{symmetric reduction} $\Pi^\sym(\mF,i,j)$ of $\mF$ with respect to a symmetric placement $P=\{(\diag_i\cap\mF)_j,(\diag_i\cap\mF)_{i-j+1}\}$ is the Ferrers sub-diagram of $\mF$ obtained by removing the set
   \begin{equation*}
       \atk(P)\cup\{(\diag_i\cap\mF)_t:t\in\{j+1,\ldots,i-j\}\}\cup \{(\diag_s\cap\mF)_t:s>i, t\in\diag_s\cap\mF\}
   \end{equation*}
    from $\widehat\mF$ and aligning the remaining cells to the top and then to the left. 
\end{definition}

Note that in Definition~\ref{def:symred} we reduce with respect to \textit{two} elements. This is necessary for guaranteeing that the reduced diagram is again symmetric.
For example, let
$\mF:=[5,5,3,2,2]$ be a symmetric Ferrers diagram. The cells marked with ``\textcolor{red}{$\times$}'' in Figure~\ref{fig:symred} are the ones deleted when constructing the 
    reduction on $\mF$ with respect to the placement $P:=\{(\diag_5\cap\mF)_2,(\diag_5\cap\mF)_4\}$. The diagram $\Pi^\sym(\mF,5,2)$ resulting from this operation is shown in Figure~\ref{fig:symredresult}.
    \begin{figure}[h]
    \begin{minipage}{0.47\textwidth}
    \centering
    \begin{tikzpicture}[scale=0.6, line width=1pt]
     \draw (0,0) grid (5,-5);
    \foreach \n in {0,...,4} {\draw[draw=black, fill =Goldenrod!55] (0,-\n) rectangle (1,-1-\n);}
    \foreach \n in {0,...,4} {\draw[draw=black, fill =Goldenrod!55] (1,-\n) rectangle (2,-1-\n);}
    \foreach \n in {0,...,2} {\draw[draw=black, fill =Goldenrod!55] (2,-\n) rectangle (3,-1-\n);}
    \foreach \n in {0,...,1} {\draw[draw=black, fill =Goldenrod!55] (3,-\n) rectangle (4,-1-\n);}
    \foreach \n in {0,...,1} {\draw[draw=black, fill =Goldenrod!55] (4,-\n) rectangle (5,-1-\n);}
    \node at (3.5,-1.5) {\color{black}\rook};
    \node at (1.5,-3.5) {\color{black}\rook};
     \node at (4.5,-1.5) {$\color{Red}\times$};
     \node at (1.5,-4.5) {$\color{Red}\times$};
     \node at (2.5,-2.5) {$\color{Red}\times$};
    \node at (0.5,-3.5) {$\color{Red}\times$};
     \node at (1.5,-2.5) {$\color{Red}\times$};
    \node at (1.5,-1.5) {$\color{Red}\times$};
    \node at (1.5,-0.5) {$\color{Red}\times$};
    \node at (0.5,-1.5) {$\color{Red}\times$};
    \node at (2.5,-1.5) {$\color{Red}\times$};
    \node at (3.5,-0.5) {$\color{Red}\times$};
    \end{tikzpicture}
    \captionof{figure}{Cells of $\mF$ deleted by the symmetric reduction.}
    \label{fig:symred}
    \end{minipage}\hfill
    \begin{minipage}{0.47\textwidth}
    \centering
    \begin{tikzpicture}[scale=0.6, line width=1pt]
     \draw (0,0) grid (5,-5);
    \foreach \n in {0,...,2} {\draw[draw=black, fill =Goldenrod!55] (0,-\n) rectangle (1,-1-\n);}
    \foreach \n in {0,...,0} {\draw[draw=black, fill =Goldenrod!55] (1,-\n) rectangle (2,-1-\n);}
    \foreach \n in {0,...,0} {\draw[draw=black, fill =Goldenrod!55] (2,-\n) rectangle (3,-1-\n);}
    \end{tikzpicture}
    \captionof{figure}{The reduced Ferrers diagram $\Pi^\sym(\mF,5,2)$.}
    \label{fig:symredresult}
    \end{minipage}
    \end{figure}

The following is the analogue of Theorem~\ref{thm:bijection} for symmetric and alternating Ferrers diagrams with symmetric and alternating placements of non-attacking rooks. We omit the proof as it is similar to the one of Theorem~\ref{thm:bijection}.

\begin{theorem}
Let $\mF$ and $\mF'$ be diagonally equivalent symmetric Ferrers diagrams. For all non-negative integers $r,s,t,\ell$ we have  
\begin{align*}
    |\{P \in \NAR^\sym(\mF,t,s) : \inv(P) = \ell \}| &= |\{P  \in \NAR^\sym(\mF',t,s) : \inv(P ) = \ell \}|, \\
|\{P \in \NAR^\alt(\mF, r) : \inv(P) = \ell \}| &= |\{P  \in \NAR^\alt(\mF', r) : \inv(P ) = \ell \}|.
\end{align*} 
In particular, diagonally equivalent symmetric Ferrers diagrams have the same $r$-th symmetric and alternating $q$-rook  polynomial.
\end{theorem}

It immediately follows that the symmetric and alternating $q$-rook polynomials of a symmetric Ferrers diagram $\mF$ are fully determined by the ordered set $\{|\diag_i\cap\mF|:i\in\N\}$.

Let  $\ell\in\N$ and  $d_1,\ldots,d_\ell\in\N$. Similarly to what is stated in Remark~\ref{rem:ferrseq}, one can see that a symmetric Ferrers diagram $\mF$ with the property $|\Delta_i\cap\mF|=d_i$, for all $i\in\{1,\ldots,\ell\}$, exists if and only if  
$(d_1,\ldots,d_\ell)$ is a Ferrers sequence and and there exist $r\in\{1,\ldots,\ell\}$ such that 
$d_i$ is even
whenever 
$i\in\{1,\ldots,r\}$ is even or when $i\in\{r+1,\ldots,\ell\}$.
We call a Ferrers sequence
$(d_1,\ldots,d_\ell)$ 
with this property \textbf{symmetric}. We denote by $\mS^\sym_\ell$ the set of symmetric Ferrers sequences of length $\ell$.

The following results are the analogues of Theorem~\ref{thm:recursive} for  alternating and symmetric $q$-rook polynomials. We omit the proofs since 
they are similar to the one of Theorem~\ref{thm:recursive}, where the symmetric reduction of Definition~\ref{def:symred} is used instead. We use the convention $R_{t,s}(0,d_1,\ldots,d_\ell)=R_{t,s}(d_1,\ldots,d_\ell)$ for non-negative integers $\ell$ and~$d_1,\ldots, d_\ell$.

\begin{theorem}\label{thm:qrookaltdiag}
Let $r\in\Z$, $\ell\in\N$, and let $(d_1,\ldots,d_\ell)\in\mS^\sym_\ell$ be a symmetric Ferrers sequence. Define $a_\ell=0$ if $d_\ell$ even and $a_\ell=1$ if $d_\ell$ odd. We have
\begin{multline*}
R_r^\alt(d_1,\ldots,d_\ell;q)=\\R_r^\alt(d_1,\ldots,d_{\ell-1};q)q^{d_\ell-a_\ell}+\sum_{j=1}^{\left\lfloor\frac{d_\ell}{2}\right\rfloor}R_{r-2}^\alt(d_3-2,\ldots,d_{\ell-1}-2,2j-2;q)q^{d_\ell-2j-a_\ell}
\end{multline*}
if $0\leq r\leq\partial(d_1,\ldots,d_\ell)$ is even and $ R_{r}^\alt(d_1,\ldots,d_\ell;q)=0$ otherwise.
\end{theorem}

\begin{theorem}\label{thm:qrooksymdiag}
Let $s,t\in\Z$, $\ell\in\N$ and let $(d_1,\ldots,d_\ell)\in\mS^\sym_\ell$ be a symmetric Ferrers sequence. Define $a_\ell=0$ if $d_\ell$ even and $a_\ell=1$ if $d_\ell$ odd. We have
\begin{multline*}
    R_{t,s}^\sym(d_1,\ldots,d_\ell;q)= R_{t,s}^\sym(d_1,\ldots,d_{\ell-1};q)q^{d_\ell}+a_\ell R_{t,s-1}^\sym(d_2-1,\ldots,d_{\ell}-1;q)\\
    +\sum_{j=1}^{\left\lfloor\frac{d_\ell}{2}\right\rfloor}R_{t-1,s}^\sym(d_3-2,\ldots,d_{\ell-1}-2,2j-2;q)q^{d_\ell-2j}
\end{multline*}
if $0\leq 2t+s\leq\partial(\mF)$ and $ R_{t,s}^\sym(d_1,\ldots,d_\ell;q)=0$ otherwise.
\end{theorem}

We conclude this section with 
the analogue of
Theorem~\ref{thm:trail} for alternating $q$-rook polynomials.

\begin{theorem}
     Let $\mF$ be a symmetric Ferrers diagram and let $r\in\N_0$. We have
     \begin{equation*}
         \tau(R^\alt_{2r}(\mF;q))=\sum_{i\in\N}(\max\{0,|\diag_{2i-1}\cap\mF|-2r-1\}+\max\{0,|\diag_{2i}\cap\mF|-2r\}).
     \end{equation*}
\end{theorem}

\section{Symmetric and Alternating Matrices}

The goal of this section is to 
provide connections between enumerative results about alternating and symmetric matrices. While these two classes of matrices are different, there are curious 
connections between the number of such matrices with prescribed properties that are hard to explain with bijective proof; see e.g.~\cite{lewis2020rook}.
In this paper, we obtain these via rook theory.

\begin{notation} 
    Let $\mF$ be a symmetric Ferrers diagram associated with the symmetric Ferrers sequence $(d_1,\ldots,d_\ell)\in\N^\ell\cup\{(0,\ldots,0)\}$. Note that the intersection $\Xi\cap\mF$ only depends on the sequence $(d_1,\ldots,d_\ell)$. This is because $\mF$ is symmetric and therefore \begin{equation} \label{XiF}
        \Xi\cap \mF = \{(i,i) \in \N^2 \, : \,  d_i \textnormal{ is odd}\}.
    \end{equation} In the following, we sometimes write~$\Xi\cap(d_1,\ldots,d_\ell)$ instead of $\Xi\cap\mF$ in this case. 
\end{notation}

The main result of this section is the following theorem.

\begin{theorem}\label{thm:altsym}
Let $(d_1,\ldots,d_\ell)\in\N^\ell\cup\{(0,\ldots,0)\}$ be symmetric Ferrers sequence. If $r=\partial(d_1,\ldots,d_\ell)$ is even, then
\begin{equation*}
    q^{|\Xi\cap(d_1,\ldots,d_\ell)|} W_r(\matalt[d_1,\ldots,d_\ell]) =  W_r(\matsym[d_1,\ldots,d_\ell]).
\end{equation*}    
\end{theorem}

In order to prove this, we require some preliminaries results. The following proposition can be seen as a generalization of \cite[Proposition~3.8]{lewis2010matrices}.

\begin{proposition}\label{Waltrec}
    Let $r\geq 2$, $\ell\in\N$, and let $(d_1,\ldots,d_\ell)\in\N^\ell\cup\{(0,\ldots,0)\}$ be a symmetric Ferrers sequence. If $r$ is odd or $(d_1,\ldots,d_\ell)=(0,\ldots,0)$, then $W_r(\matalt[d_1,\ldots,d_\ell])=0$. Otherwise we have
    \begin{align*}
       W_r(\matalt[d_1,\ldots,d_\ell])&=W_r(\matalt[d_1,\ldots,d_{\ell-1}])\\&+(q-1)q^{\ell-2}\sum_{j=1}^{\left\lfloor\frac{d_\ell}{2}\right\rfloor}W_{r-2}(\matalt[d_3-2,\ldots,d_{\ell-1}-2,2j-2]).
    \end{align*}
\end{proposition}
\begin{proof}
    The result is trivial for $r$ odd or $(d_1,\ldots,d_\ell)=(0,\ldots,0)$. We assume $r$ even and $(d_1,\ldots,d_\ell)\neq(0,\ldots,0)$ in the remainder of the proof. For every $1\leq i\leq \ell$,  define 
    \begin{equation*}
        a_i=\begin{cases}
            0 & \textup{ if } d_i \textup{ even},\\
            1 & \textup{ if } d_i \textup{ odd},
        \end{cases}\qquad\textup{ and }\qquad b_\ell=\left\lfloor\frac{d_\ell}{2}\right\rfloor.
    \end{equation*}
    Observe that if $b_\ell=0$ then $ \sum_{j=1}^{b_\ell}W_{r-2}(\matalt[d_3-2,\ldots,d_{\ell-1}-2,2j-2])=0$ by definition, since the sum runs over the empty set.
    By applying Theorem~\ref{thm:qrookaltdiag} to Theorem~\ref{thm:altqrooken} we get
    \begin{multline*}
        W_r(\matalt[d_1,\ldots,d_\ell]) \\=(q-1)^{r/2}q^{(\sum_{i=1}^\ell d_i-\sum_{i=1}^\ell a_i-r)/2 }\Big(R_r^\alt(d_1,\ldots,d_{\ell-1};q^{-1/2})q^{-(d_\ell-a_\ell)/2}\\ +\sum_{j=1}^{b_\ell}R_{r-2}^\alt(d_3-2,\ldots,d_{\ell-1}-2,2j-2;q^{-1/2})q^{-(d_\ell-2j-a_\ell)/2}\Big),
    \end{multline*}
    where we used the facts that $|\mF|=\sum_{i=1}^\ell d_i$ and  $|\Xi\cap\mF|=\sum_{i=1}^\ell a_i$; see~\eqref{XiF}. It is not hard to see that 
    \begin{align*}
         &(q-1)^{r/2}q^{(\sum_{i=1}^\ell d_i-\sum_{i=1}^\ell a_i-r)/2 }R_r^\alt(d_1,\ldots,d_{\ell-1};q^{-1/2})q^{-(d_\ell-a_\ell)/2}\\
         &=(q-1)^{r/2}q^{(\sum_{i=1}^{\ell-1} d_i-\sum_{i=1}^{\ell-1} a_i-r)/2 }R_r^\alt(d_1,\ldots,d_{\ell-1};q^{-1/2})\\
         &=W_r(\matalt[d_1,\ldots,d_{\ell-1}]).
    \end{align*}
    On the other hand, for any $1\leq j\leq b_\ell$, we have 
    \begin{multline}\label{eq:Wraltrec}
    (q-1)^{r/2}q^{(\sum_{i=1}^\ell d_i-\sum_{i=1}^\ell a_i-r)/2}\sum_{j=1}^{b_\ell}R_{r-2}^\alt(d_3-2,\ldots,d_{\ell-1}-2,2j-2;q^{-1/2})q^{-(d_\ell-2j-a_\ell)/2}\\
    =\sum_{j=1}^{b_\ell}(q-1)^{r/2}q^{(\sum_{i=1}^{\ell-1} d_i+2j-\sum_{i=1}^{\ell-1} a_i-r)/2}R_{r-2}^\alt(d_3-2,\ldots,d_{\ell-1}-2,2j-2;q^{-1/2}).
    \end{multline}
    Note that if $d_3\leq 1$ then this sum is $0$, and we can assume $d_3\in\{2,3\}$ in the remainder of the proof. Observe that if $d_3=2$ then the only possible symmetric Ferrers sequence is $(1,2,2,\ldots,2)$, $b_\ell=1$ and we have
    \begin{equation*}
        R_{r-2}^\alt(d_3-2,\ldots,d_{\ell-1}-2,2-2;q^{-1/2})=R_{r-2}^\alt(0,\ldots,0;q^{-1/2})
    \end{equation*}
    which is $1$ if $r=2$ and $0$ otherwise.
    Observe also that $d_3\in\{2,3\}$ implies $d_1=1$ and $d_2=2$ since the diagram is symmetric, from which $a_1=1$ and $a_2=0$. 
    Further note that we have
    \begin{align*}
        \sum_{i=1}^{\ell-1} d_i+2j-\sum_{i=1}^{\ell-1} a_i-r&=d_1+d_2+\sum_{i=3}^{\ell-1} (d_i-2)+\sum_{i=3}^{\ell-1}2+2j-2+2-a_1-a_2-\sum_{i=3}^{\ell-1} a_i-r\\
        &=3+\sum_{i=3}^{\ell-1} (d_i-2)+2(\ell-3)+(2j-2)-1-\sum_{i=3}^{\ell-1} a_i-(r-2)\\
        &=\sum_{i=3}^{\ell-1} (d_i-2)+(2j-2)-\sum_{i=3}^{\ell-1} a_i-(r-2)+2(\ell-2).
    \end{align*}
    Therefore, we can rewrite as follows
    \begin{align*}
        &(q-1)^{r/2}q^{(\sum_{i=1}^\ell d_i-\sum_{i=1}^\ell a_i-r)/2}R_{r-2}^\alt(d_3-2,\ldots,d_{\ell-1}-2,2j-2;q^{-1/2})q^{-(d_\ell-2j-a_\ell)/2}\\
        &=q^{\ell-2}(q-1)^{r/2}q^{(\sum_{i=3}^{\ell-1} (d_i-2)+(2j-2)-\sum_{i=3}^{\ell-1} a_i-(r-2))/2}R_{r-2}^\alt(d_3-2,\ldots,d_{\ell-1}-2,2j-2;q^{-1/2})\\
        &=(q-1)q^{\ell-2}(q-1)^{(r-2)/2}q^{(\sum_{i=3}^{\ell-1} (d_i-2)+(2j-2)-\sum_{i=3}^{\ell-1} a_i-(r-2))/2}\cdot \\
        &\quad \quad  R_{r-2}^\alt(d_3-2,\ldots,d_{\ell-1}-2,2j-2;q^{-1/2})\\
        &=(q-1)q^{\ell-2}W_{r-2}(\matalt[d_3-2,\ldots,d_{\ell-1}-2,2j-2]),
    \end{align*}
    which applied to~\eqref{eq:Wraltrec} concludes the proof.
\end{proof}

The next result generalizes~\cite[Lemma~4]{macwilliams1969orthogonal}.
Its proof combines
Theorem~\ref{thm:qrooksymdiag} 
with Theorem~\ref{thm:symcount}.

\begin{proposition}\label{Wsymrec}
      Let $r\geq 2$ and let $(d_1,\ldots,d_\ell)\in\mS^\sym_\ell$ be a symmetric Ferrers sequence.
      Define $a_\ell=0$ if $d_\ell$ even, and $a_\ell=1$ if $d_\ell$ odd. If $(d_1,\ldots,d_\ell)=(0,\ldots,0)$ then $W_r(\matsym[d_1,\ldots,d_\ell])=0$. Otherwise we have
     \begin{align*}
     W_r(\matsym[d_1,\ldots,d_\ell])&= W_r(\matsym[d_1,\ldots,d_{\ell-1}])\\&+a_\ell (q-1)q^{(\ell-1)/2} W_{r-1}(\matsym[d_2-1,\ldots,d_\ell-1])\\&+(q-1)q^{\ell-1}\sum_{j=1}^{\left\lfloor\frac{d_\ell}{2}\right\rfloor}W_{r-2}(\matsym[d_3-2,\ldots,d_{\ell-1}-2,2j-2]).
     \end{align*}
\end{proposition}
\begin{proof}
    We use the same notation as in the proof of Proposition~\ref{Waltrec}. The result is trivial for $(d_1,\ldots,d_\ell)=(0,\ldots,0)$ and we can assume $(d_1,\ldots,d_\ell)\neq(0,\ldots,0)$ in the remainder of the proof. By combining Theorem~\ref{thm:qrooksymdiag} and Theorem~\ref{thm:symcount}, we obtain
    \begin{align*}
W_r(\matalt[d_1,\ldots,d_\ell])&=\sum_{\substack{t,s \ge 0 \\ 2t+s=r} } (q-1)^{t+s} q^{(\sum_{i=1}^\ell d_i-t-s)/2}\Big(R_{t,s}^\sym(d_1,\ldots,d_{\ell-1};q^{-1/2})q^{-d_\ell/2}\\&+a_\ell R_{t,s-1}^\sym(d_2-1,\ldots,d_{\ell}-1;q^{-1/2})\\
        &+\sum_{j=1}^{b_\ell}R_{t-1,s}^\sym(d_3-2,\ldots,d_{\ell-1}-2,2j-2;q^{-1/2})q^{-(d_\ell-2j)/2}\Big).
    \end{align*}
    Notice that since $(d_1,\ldots,d_\ell)\neq(0,\ldots,0)$, we have $d_1=1$ and therefore
    \begin{align*}
        &\sum_{\substack{t,s \ge 0 \\ 2t+s=r} } (q-1)^{t+s} q^{(\sum_{i=1}^\ell d_i-t-s)/2}R_{t,s}^\sym(d_1,\ldots,d_{\ell-1};q^{-1/2})q^{-d_\ell/2}\\
        &=\sum_{\substack{t,s \ge 0 \\ 2t+s=r} } (q-1)^{t+s} q^{(\sum_{i=1}^{\ell-1} d_i-t-s)/2}R_{t,s}^\sym(d_1,\ldots,d_{\ell-1};q^{-1/2})\\
        &= W_r(\matsym[d_1,\ldots,d_{\ell-1}]).
    \end{align*}
    Similarly,
    \begin{align*}
        &\sum_{\substack{t,s \ge 0 \\ 2t+s=r} } (q-1)^{t+s} q^{(\sum_{i=1}^\ell d_i-t-s)/2}a_\ell R_{t,s-1}^\sym(d_2-1,\ldots,d_{\ell}-1;q^{-1/2})\\
        &=a_\ell(q-1)\sum_{\substack{t,s \ge 0 \\ 2t+s=r} }(q-1)^{t+(s-1)}q^{(d_1+\sum_{i=2}^\ell (d_i-1)+\sum_{i=2}^\ell 1 -t-s)/2}R_{t,s-1}^\sym(d_2-1,\ldots,d_{\ell}-1;q^{-1/2})\\
        &=a_\ell(q-1)\sum_{\substack{t,s \ge 0 \\ 2t+s=r} }(q-1)^{t+(s-1)}q^{(\sum_{i=2}^\ell (d_i-1)+(\ell-1) -t-(s-1))/2}R_{t,s-1}^\sym(d_2-1,\ldots,d_{\ell}-1;q^{-1/2})\\
        &=a_\ell(q-1)\sum_{\substack{t,s \ge 0 \\ 2t+s=r-1} }(q-1)^{t+s}q^{(\sum_{i=2}^\ell (d_i-1)+(\ell-1) -t-s)/2}R_{t,s}^\sym(d_2-1,\ldots,d_{\ell}-1;q^{-1/2})\\
        &=a_\ell(q-1)q^{(\ell-1)/2}W_{r-1}(\matsym[d_2-1,\ldots,d_\ell-1]).
    \end{align*}
    If $d_3\leq 1$, then $b_\ell=0$ and
    \begin{equation*}
        \sum_{j=1}^{b_\ell}R_{t-1,s}^\sym(d_3-2,\ldots,d_{\ell-1}-2,2j-2;q^{-1/2})q^{-(d_\ell-2j)/2}=0.
    \end{equation*}
    Therefore, we can assume $d_3\in\{2,3\}$ in the remainder of the proof. Moreover, similarly to what was observed in the proof of Proposition~\ref{Waltrec}, if $d_3=0$ then the only admissible symmetric Ferrers sequence is $(1,2,2,\ldots,2)$, $b_\ell=1$, and we have
    \begin{equation*}
        R_{t-1,s}^\sym(d_3-2,\ldots,d_{\ell-1}-2,2j-2;q^{-1/2})=R_{t-1,s}^\sym(0,\ldots,0;q^{-1/2}),
    \end{equation*}
    which is $1$ if and only if $t=1$ and $s=0$. Finally, note that $d_3\in\{2,3\}$ implies $d_1=1$ and $d_2=2$. Thus
    \allowdisplaybreaks
    \begin{align*}
        &\sum_{\substack{t,s \ge 0 \\ 2t+s=r} } (q-1)^{t+s} q^{(\sum_{i=1}^\ell d_i-t-s)/2}\sum_{j=1}^{b_\ell}R_{t-1,s}^\sym(d_3-2,\ldots,d_{\ell-1}-2,2j-2;q^{-1/2})q^{-(d_\ell-2j)/2}\\
        &=\sum_{j=1}^{b_\ell}(q-1)\sum_{\substack{t,s \ge 0 \\ 2t+s=r}}(q-1)^{(t-1)+s}q^{(d_1+d_2+\sum_{i=3}^{\ell-1} d_i+2j-t-s)/2} \\
        &\hspace{8cm}\cdot R_{t-1,s}^\sym(d_3-2,\ldots,d_{\ell-1}-2,2j-2;q^{-1/2})\\
        &=\sum_{j=1}^{b_\ell}(q-1)\sum_{\substack{t,s \ge 0 \\ 2t+s=r}}(q-1)^{(t-1)+s}q^{(\sum_{i=3}^{\ell-1} (d_i-2)+(2j-2)+2(\ell-1)-(t-1)-s)/2}\\ &\hspace{8cm}\cdot R_{t-1,s}^\sym(d_3-2,\ldots,d_{\ell-1}-2,2j-2;q^{-1/2})\\
        &=\sum_{j=1}^{b_\ell}(q-1)q^{\ell-1}\sum_{\substack{t,s \ge 0 \\ 2t+s=r}}(q-1)^{(t-1)+s}q^{(\sum_{i=3}^{\ell-1} (d_i-2)+(2j-2)-(t-1)-s)/2}\\ &\hspace{8cm} \cdot R_{t-1,s}^\sym(d_3-2,\ldots,d_{\ell-1}-2,2j-2;q^{-1/2})\\
        &=\sum_{j=1}^{b_\ell}(q-1)q^{\ell-1}\sum_{\substack{t,s \ge 0 \\ 2t+s=r-2}}(q-1)^{t+s}q^{(\sum_{i=3}^{\ell-1} (d_i-2)+(2j-2)-t-s)/2}\\
        &\hspace{8cm}\cdot R_{t,s}^\sym(d_3-2,\ldots,d_{\ell-1}-2,2j-2;q^{-1/2})\\
        &=\sum_{j=1}^{b_\ell}(q-1)q^{\ell-1}W_{r-2}(\matsym[d_3-2,\ldots,d_{\ell-1}-2,2j-2]),
    \end{align*}
    This concludes the proof.
\end{proof}

\begin{remark}\label{rem:Wsymevenodd}
    Let $(d_1,\ldots,d_\ell)$ be a symmetric Ferrers sequence and assume $d_\ell\geq 1$ odd, that is, $d_\ell-1\geq 0$ even. Then Proposition~\ref{Wsymrec} implies
         \begin{align*}
        W_r(\matsym[d_1,\ldots,d_\ell-1])&= W_r(\matsym[d_1,\ldots,d_{\ell-1}])\\&+(q-1)q^{\ell-1}\sum_{j=1}^{\left\lfloor\frac{d_\ell}{2}\right\rfloor}W_{r-2}(\matsym[d_3-2,\ldots,d_{\ell-1}-2,2j-2]).
     \end{align*}
     and therefore we have
          \begin{align*}
        W_r(\matsym[d_1,\ldots,d_\ell])&= W_r(\matsym[d_1,\ldots,d_{\ell-1}])\\&+a_\ell (q-1)q^{(\ell-1)/2} W_{r-1}(\matsym[d_2-1,\ldots,d_\ell-1])\\&+(q-1)q^{\ell-1}\sum_{j=1}^{\left\lfloor\frac{d_\ell}{2}\right\rfloor}W_{r-2}(\matsym[d_3-2,\ldots,d_{\ell-1}-2,2j-2])\\
        &=W_r(\matsym[d_1,\ldots,d_{\ell}-1])\\&+a_\ell (q-1)q^{(\ell-1)/2} W_{r-1}(\matsym[d_2-1,\ldots,d_\ell-1]).
     \end{align*}
     We will use this identity later in the proof of Theorem~\ref{thm:altsym}.
\end{remark}

The proof of
Theorem~\ref{thm:altsym}
relies on two more preparatory results.
We establish these
using an induction argument on posets,
sometimes called \textit{well-founded induction}.
We refer the reader to~\cite{takeuti2012introduction}, for instance.

\begin{definition}
   A poset $(S,\leq)$ 
   is \textbf{well-founded} if every descending chain of elements is finite, that is, there are no infinite descending chains in $S$ with respect to $\leq$.
\end{definition}

The well-founded induction principle
for a well-founded poset $(S,\leq)$
and a property $P$ defined on the elements of $S$:
\begin{equation*}
    ((P(y)\;\forall\; y \in S \mbox{ with } y \le x) \Longrightarrow P(x) \;\forall\; x\in S) \Longrightarrow P(x) \;\forall\; x\in S.
\end{equation*}
In words, if for all $x\in S$, whenever $P(y)$ holds for all $y\in S$ that \textit{precede} $x$ implies that $P(x)$ holds, then $P(x)$ holds for all $x\in S$.

One can check that, for any $\ell\in\N$, the set of 
symmetric Ferrers sequences $\mS_\ell^\sym$
with 
the product order $\leq_\sym$
is a well-founded poset.

\begin{remark}\label{rem:partial12}
    Let $\ell\in\N$ and let $(d_1,\ldots,d_\ell)\in\mS^\sym_\ell$ be a symmetric Ferrers sequence. It is not difficult to check that if $d_\ell\geq 1$, then $\partial(d_2-1,\ldots,d_\ell-1)\leq \partial(d_1,\ldots,d_\ell)-1$.
    Moreover, if $d_\ell\geq 2$, then $\partial(d_3-2,\ldots,d_\ell-2)\leq \partial(d_1,\ldots,d_\ell)-2$.
\end{remark}

We can now state and prove the last two preparatory lemmas.

\begin{lemma}\label{lem:altsymeven}
    Let $r\geq 2$ and $(d_1,\ldots,d_\ell)\in\mS_\ell^\sym$ be symmetric Ferrers sequence. If $r=\partial(d_1,\ldots,d_\ell)$  and $d_\ell$ are even, then
\begin{equation*}
    q^{|\Xi\cap(d_1,\ldots,d_\ell)|} \, W_r(\matalt[d_1,\ldots,d_\ell]) =  W_r(\matsym[d_1,\ldots,d_\ell]).
\end{equation*}    
\end{lemma}
\begin{proof}
    We use the same notation as in the proof of Proposition~\ref{Waltrec}. One can check that $W_2(\matsym[1,2])=q(q-1)=qW_2(\matalt[1,2])$. We use the well-founded induction principle (and Remark~\ref{rem:partial12}),
    assuming that the statement is true for all symmetric sequences that precede $(d_1,\ldots,d_\ell)$ in the poset $(\mS_\ell^\sym,\leq_\sym)$. Since $d_\ell$ is even, $|\Xi\cap(d_1,\ldots,d_\ell)|= |\Xi\cap(d_1,\ldots,d_{\ell-1})|$. Moreover, since $r=\partial(\mF)$ is even and $d_1=1$, by Propositions~\ref{Waltrec} and~\ref{Wsymrec} we have 
    \begin{align*}        W_r(\matsym&[d_1,\ldots,d_{\ell}])\\ &=W_r(\matsym[d_1,\ldots,d_{\ell-1}])\\&\qquad +(q-1)q^{\ell-1}\sum_{j=1}^{\left\lfloor\frac{d_\ell}{2}\right\rfloor}W_{r-2}(\matsym[d_3-2,\ldots,d_{\ell-1}-2,2j-2])\\
        &=q^{|\Xi\cap(d_1,\ldots,d_{\ell-1})|}W_r(\matalt[d_1,\ldots,d_{\ell-1}])\\&\qquad +(q-1)q^{\ell-1}q^{|\Xi\cap(d_3,\ldots,d_{\ell-1})|}\sum_{j=1}^{\left\lfloor\frac{d_\ell}{2}\right\rfloor}W_{r-2}(\matalt[d_3-2,\ldots,d_{\ell-1}-2,2j-2])\\
        &=q^{|\Xi\cap(d_1,\ldots,d_{\ell-1})|}W_r(\matalt[d_1,\ldots,d_{\ell-1}])\\&\qquad +(q-1)q^{\ell-2}q^{|\Xi\cap(d_1,\ldots,d_{\ell-1})|}\sum_{j=1}^{\left\lfloor\frac{d_\ell}{2}\right\rfloor}W_{r-2}(\matalt[d_3-2,\ldots,d_{\ell-1}-2,2j-2])\\
        &=q^{|\Xi\cap(d_1,\ldots,d_{\ell-1})|}W_r(\matalt[d_1,\ldots,d_{\ell}])\\
        &=q^{|\Xi\cap(d_1,\ldots,d_{\ell})|}W_r(\matalt[d_1,\ldots,d_{\ell}]).
    \end{align*}
    This concludes the proof.
\end{proof}

The following result generalizes~\cite[Proposition~3.6]{lewis2010matrices}, which can be 
recovered for the symmetric Ferrers sequence $(1,2,\ldots,n-1,n,n-1,\ldots,2,1)$, $n\in\N$.

\begin{lemma}\label{lem:dellodd}
   Let $r\geq 2$ and let $(d_1,\ldots,d_\ell)\in\mS^\sym_\ell$ be a symmetric Ferrers sequence. If $r=\partial(d_1,\ldots,d_\ell)$ is even and $d_\ell$ is odd, then
   \begin{equation*}
       W_{r}(\matalt[d_1,\ldots,d_\ell])= W_{r-1}(\matsym[d_2-1,\ldots,d_\ell-1]).
   \end{equation*}
\end{lemma}
\begin{proof}
Since the Ferrers sequence is symmetric, $d_\ell$ odd implies that $|\Xi\cap(d_1,\ldots,d_\ell)\cap\Delta_\ell|=1$. Therefore, since alternating matrices have zeros on the main diagonal, the number of rank $r$ alternating matrices supported on $(d_1,\ldots,d_\ell)$ is the same as the number of rank $r$ alternating matrices supported on $(d_1,\ldots,d_\ell-1)$. Moreover, 
$r-1=\partial(d_2-1,\ldots,d_\ell-1)$ by Proposition~\ref{prop:dFdiag}. It easy to check that $W_{2}(\matalt[1])=0=W_{1}(\matsym[0])$ and  $W_{2}(\matalt[1,2,1])=q-1=W_{1}(\matsym[1])$. These are all the possible cases for $r=2$ and thus we can assume $r\geq 4$ in the remainder of the proof. In order for the alternating matrices to have rank at least $4$, it holds that $d_4=4$. We use the well-founded induction principle (and Remark~\ref{rem:partial12}), assuming that the statement is true for all symmetric sequences that precede $(d_1,\ldots,d_\ell)$ in the poset $(\mS_\ell^\sym,\leq_\sym)$. We distinguish two cases depending to the value of $d_\ell$.

\noindent\underline{Case 1}: Assume that $d_\ell=1$. By Proposition~\ref{Wsymrec} and the fact that $d_{\ell-1}-1\geq 1$ is odd, we obtain
\begin{align} \label{eq:recsym}
    &W_{r-1}(\matsym[d_2-1,\ldots,d_{\ell-1}-1,0]) \nonumber \\&=W_{r-1}(\matsym[d_2-1,\ldots,d_{\ell-1}-1]) \nonumber
     \\&=W_{r-1}(\matsym[d_2-1,\ldots,d_{\ell-2}-1]) 
    \\&\qquad +(q-1)q^{(\ell-3)/2}W_{r-2}(\matsym[d_3-2,\ldots,d_{\ell-1}-2]) \nonumber
    \\&\qquad +(q-1)q^{\ell-3} \sum_{j=1}^{(d_{\ell-1}-2)/2} W_{r-3}(\matsym[d_4-3,\ldots,d_{\ell-1}-3,2j-2]). \nonumber
\end{align}
Now, applying Lemma~\ref{lem:altsymeven} to Equation~\eqref{eq:recsym} and assuming the statement is true for all symmetric Ferrers sequences preceding $(d_1,\ldots,d_{\ell-1},1)$ in $(\mS_\ell^\sym,\leq_\sym)$, we have
\allowdisplaybreaks
\begin{align*}
    W_{r-1}(\matsym&[d_2-1,\ldots,d_{\ell-1}-1,0])
    \\&=W_{r}(\matalt[1,d_2,\ldots,d_{\ell-2}])
    \\&\qquad +(q-1)q^{(\ell-3)/2}q^{(\ell-3)/2}W_{r-2}(\matalt[d_3-2,\ldots,d_{\ell-1}-2])
     \\&\qquad +(q-1)q^{\ell-3} \sum_{j=1}^{(d_{\ell-1}-2)/2} W_{r-2}(\matalt[1,d_4-2,\ldots,d_{\ell-1}-2,2j-1])
     \\&=W_{r}(\matalt[1,d_2,\ldots,d_{\ell-2}])
    \\&\qquad +(q-1)q^{(\ell-3)/2}q^{(\ell-3)/2}W_{r-2}(\matalt[1,d_4-2,\ldots,d_{\ell-1}-2])
     \\&\qquad +(q-1)q^{\ell-3} \sum_{j=1}^{(d_{\ell-1}-2)/2} W_{r-2}(\matalt[1,d_4-2,\ldots,d_{\ell-1}-2,2j-1])
     \\&=W_{r}(\matalt[1,d_2,\ldots,d_{\ell-2}])
     \\&\qquad +(q-1)q^{\ell-3} \sum_{j=0}^{(d_{\ell-1}-2)/2} W_{r-2}(\matalt[1,d_4-2,\ldots,d_{\ell-1}-2,2j-1])
     \\&=W_{r}(\matalt[1,d_2,\ldots,d_{\ell-2}])
     \\&\qquad +(q-1)q^{\ell-3} \sum_{j=1}^{d_{\ell-1}/2} W_{r-2}(\matalt[1,d_4-2,\ldots,d_{\ell-1}-2,2j-2])
     \\&=W_{r}(\matalt[d_1,\ldots,d_\ell-1])
     \\&=W_{r}(\matalt[d_1,\ldots,d_\ell-1,1]),
\end{align*}
where the second-to-last equality follows from Proposition~\ref{Waltrec}.

\noindent\underline{Case 2:} Assume that $d_\ell\geq 3$. Proposition~\ref{Waltrec} gives
    \begin{equation}\label{eq1:dellodd}
        \begin{aligned}
             W_{r}(\matalt[d_1,\ldots,d_\ell])&= W_{r}(\matalt[d_1,\ldots,d_\ell-1])\\
             &= W_{r}(\matalt[d_1,\ldots,d_{\ell-1}]) \\&\qquad +
             (q-1)q^{\ell-2}\sum_{j=1}^{(d_\ell-1)/2} W_{r-2}(\matalt[d_3-2,\ldots,d_{\ell-1}-2,2j-2]).
        \end{aligned}
    \end{equation}
    On the other hand, Theorem~\ref{Wsymrec} and the fact that $d_\ell-1$ is even imply
    \begin{multline}\label{eq2:dellodd}
    W_{r-1}(\matsym[d_2-1,\ldots,d_\ell-1]) =W_{r-1}(\matsym[d_2-1,\ldots,d_{\ell-1}-1])+\\
             (q-1)q^{(\ell-1)-1} \sum_{j=1}^{(d_\ell-1)/2} W_{r-3}(\matsym[d_4-3,\ldots,d_{\ell-1}-3,2j-2]).
    \end{multline}
    One checks that $W_{r}(\matalt[d_1,\ldots,d_\ell])$ and $  W_{r-1}(\matsym[d_2-1,\ldots,d_\ell-1])$ satisfy the same recurrence relation, given in Equations~\eqref{eq1:dellodd} and~\eqref{eq2:dellodd}. Since $W_2(\matalt[1,2,3])=W_2(\matalt[1,2])=q=W_2(\matsym[1])=W_1(\matsym[1,2])$, the two relations have the same initial conditions. This proves the statement of the lemma.
\end{proof}

We are finally ready to prove Theorem~\ref{thm:altsym}.

\begin{proof}[Proof of Theorem~\ref{thm:altsym}]
        Recall  that $|\Xi\cap(d_1,\ldots,d_\ell)|=|\{1\leq i\leq\ell \, :\, d_i \textup{ odd}\}|$; see~\eqref{XiF}. The result is trivial for $r=0$ or $(d_1,\ldots,d_\ell)=(0,\ldots,0)$. Indeed, we have 
        \begin{enumerate}
            \item $W_0(\matsym[0])=1=W_0(\matalt[0])$,
            \item $W_r(\matsym[0])=0=W_r(\matalt[0])$ for all $r\geq 1$,
            \item $W_0(\matsym[d_1,\ldots,d_\ell])=1=W_0(\matalt[d_1,\ldots,d_\ell])$ for all $(d_1,\ldots,d_\ell)\in\N^\ell$.
            \end{enumerate}
        We assume $r\geq 2$ and $(d_1,\ldots,d_\ell)\neq(0,\ldots,0)$. Using Propositions~\ref{Waltrec} and~\ref{Wsymrec} and an induction argument, it is immediate to check that if the result holds for the symmetric Ferrers sequence $(d_1,\ldots,d_\ell)$ then it holds for the symmetric Ferrers sequence $(d_1,\ldots,d_\ell,0)$. Therefore it suffices to prove the result for $d_\ell\neq 0$.  The case where $d_\ell$ is even is given by Lemma~\ref{lem:altsymeven}. For $d_\ell$ odd 
        we have $\ell$ odd and so $d_{\ell-1}$ even. By Propositions~\ref{Waltrec} and~\ref{Wsymrec}, Remark~\ref{rem:Wsymevenodd}, Lemma~\ref{lem:dellodd}, and the first part of the proof, we have 
        \allowdisplaybreaks
    \begin{align*}
W_r(\matsym[d_1,\ldots,d_{\ell}])&=W_r(\matsym[d_1,\ldots,d_{\ell-1}])\\&\qquad +(q-1)q^{\ell-1}\sum_{j=1}^{\left\lfloor\frac{d_\ell}{2}\right\rfloor}W_{r-2}(\matsym[d_3-2,\ldots,d_{\ell-1}-2,2j-2])\\&\qquad +(q-1)q^{(\ell-1)/2} W_{r-1}(\matsym[d_2-1,\ldots,d_\ell-1])\\
&=W_r(\matsym[d_1,\ldots,d_{\ell}-1])\\&\qquad +(q-1)q^{(\ell-1)/2} W_{r-1}(\matsym[d_2-1,\ldots,d_\ell-1])\\
        &=q^{|\Xi\cap(d_1,\ldots,d_{\ell-1})|}W_r(\matalt[d_1,\ldots,d_{\ell}-1])\\&\qquad +(q-1)q^{(\ell-1)/2} W_{r-1}(\matsym[d_2-1,\ldots,d_\ell-1])\\
        &=q^{(\ell-1)/2}W_r(\matalt[d_1,\ldots,d_{\ell}])\\&\qquad +(q-1)q^{(\ell-1)/2} W_r(\matalt[d_1,\ldots,d_{\ell}])\\
        &=q^{(\ell+1)/2}W_r(\matalt[d_1,\ldots,d_{\ell}])\\
        &=q^{|\Xi\cap(d_1,\ldots,d_\ell)|}W_r(\matalt[d_1,\ldots,d_{\ell}]).
    \end{align*}
    Note that we used that $|\Xi\cap(d_1,\ldots,d_{\ell-1})|=(\ell-1)/2$ and $|\Xi\cap(d_1,\ldots,d_{\ell})|=(\ell+1)/2$, which is a consequence of the fact that $d_\ell$ is odd and $d_{\ell-1}$ is even; see~\eqref{XiF}. This concludes the proof.
\end{proof}

\bigskip

\bibliographystyle{amsplain}
\bibliography{ourbib}

\providecommand{\bysame}{\leavevmode\hbox to3em{\hrulefill}\thinspace}
\providecommand{\MR}{\relax\ifhmode\unskip\space\fi MR }
\providecommand{\MRhref}[2]{%
  \href{http://www.ams.org/mathscinet-getitem?mr=#1}{#2}
}
\providecommand{\href}[2]{#2}
\begin{thebibliography}{10}

\bibitem{etzion2009error}
T.~Etzion and N.~Silberstein, \emph{Error-correcting codes in projective spaces via rank-metric codes and {F}errers diagrams}, IEEE Transactions on Information Theory \textbf{55} (2009), no.~7, 2909--2919.

\bibitem{GaRe86}
A.~M. Garsia and J.~B. Remmel, \emph{\mbox{$Q$}-counting rook configurations and a formula of {F}robenius}, Journal of Combinatorial Theory, Series A \textbf{41} (1986), 246--275.

\bibitem{gluesing2020partitions}
H.~Gluesing-Luerssen and A.~Ravagnani, \emph{Partitions of matrix spaces with an application to q-rook polynomials}, European Journal of Combinatorics \textbf{89} (2020), 103120.

\bibitem{gruica2022rook}
A.~Gruica and A.~Ravagnani, \emph{Rook theory of the {E}tzion-{S}ilberstein conjecture}, arXiv preprint arXiv:2209.05114 (2022).

\bibitem{haglund1998q}
J.~Haglund, \emph{q-rook polynomials and matrices over finite fields}, Advances in Applied Mathematics \textbf{20} (1998), no.~4, 450--487.

\bibitem{haglund2001rook}
J.~Haglund and J.~B. Remmel, \emph{Rook theory for perfect matchings}, Advances in Applied Mathematics \textbf{27} (2001), no.~2-3, 438--481.

\bibitem{kaplansky1946problem}
I.~Kaplansky and J.~Riordan, \emph{The problem of the rooks and its applications}, Duke Math. J. \textbf{13} (1946), no.~1, 259--268.

\bibitem{klein2014counting}
A.~J. Klein, J.~B. Lewis, and A.~H. Morales, \emph{Counting matrices over finite fields with support on skew young diagrams and complements of rothe diagrams}, Journal of Algebraic Combinatorics \textbf{39} (2014), 429--456.

\bibitem{lewis2010matrices}
J.~B. Lewis, R.~I. Liu, A.~H. Morales, G.~Panova, S.~V Sam, and Y.~X. Zhang, \emph{Matrices with restricted entries and $ q $-analogues of permutations}, Journal of Combinatorics \textbf{2} (2011), no.~3, 355--395.

\bibitem{lewis2020rook}
J.~B. Lewis and A.~H. Morales, \emph{Rook theory of the finite general linear group}, Experimental Mathematics \textbf{29} (2020), no.~3, 328--346.

\bibitem{macwilliams1969orthogonal}
J.~MacWilliams, \emph{Orthogonal matrices over finite fields}, The American Mathematical Monthly \textbf{76} (1969), no.~2, 152--164.

\bibitem{neri2023proof}
A.~Neri and M.~Stanojkovski, \emph{A proof of the {E}tzion-{S}ilberstein conjecture for monotone and mds-constructible ferrers diagrams}, arXiv preprint arXiv:2306.16407 (2023).

\bibitem{riordan1958introduction}
J.~Riordan, \emph{An introduction to combinatorial analysis}, An Introduction to Combinatorial Analysis, John Wiley, New York, 1958.

\bibitem{stembridge1998counting}
J.~R. Stembridge, \emph{Counting points on varieties over finite fields related to a conjecture of {K}ontsevich}, Annals of Combinatorics \textbf{2} (1998), no.~4, 365--385.

\bibitem{takeuti2012introduction}
G.~Takeuti and W.~M. Zaring, \emph{Introduction to axiomatic set theory}, vol.~1, Springer Science \& Business Media, 2012.

\end{thebibliography}

\end{document}